\documentclass[12pt,reqno]{amsart}
\usepackage{amsmath,amsthm,amssymb,amscd,url, color}

\begin{document}

\allowdisplaybreaks


\theoremstyle{plain}
\newtheorem{theorem}{Theorem}
\newtheorem{conjecture}[theorem]{Conjecture}
\newtheorem{proposition}[theorem]{Proposition}
\newtheorem{lemma}[theorem]{Lemma}
\newtheorem{corollary}[theorem]{Corollary}
\newtheorem{conjecturedlemma}[theorem]{Conjectured Lemma}
\newtheorem{question}[theorem]{Question}

\theoremstyle{definition}
\newtheorem{definition}[theorem]{Definition}

\theoremstyle{remark}
\newtheorem{remark}[theorem]{Remark}
\newtheorem{example}[theorem]{Example}
\newtheorem*{claim}{Claim}
\newtheorem*{acknowledgement}{Acknowledgements}

\def\BigStrut{\vphantom{$(^{(^(}_{(}$}} 

\newenvironment{notation}[0]{%
  \begin{list}%
    {}%
    {\setlength{\itemindent}{0pt}
     \setlength{\labelwidth}{4\parindent}
     \setlength{\labelsep}{\parindent}
     \setlength{\leftmargin}{5\parindent}
     \setlength{\itemsep}{0pt}
     }%
   }%
  {\end{list}}

\newenvironment{parts}[0]{%
  \begin{list}{}%
    {\setlength{\itemindent}{0pt}
     \setlength{\labelwidth}{1.5\parindent}
     \setlength{\labelsep}{.5\parindent}
     \setlength{\leftmargin}{2\parindent}
     \setlength{\itemsep}{0pt}
     }%
   }%
  {\end{list}}
\newcommand{\Part}[1]{\item[\upshape#1]}

\newcount\stepcount
\def\Step{\advance\stepcount by1\par\noindent
     \textbf{Step \the\stepcount}:\enspace\ignorespaces}

\renewcommand{\a}{\alpha}
\renewcommand{\b}{\beta}
\newcommand{\g}{\gamma}
\renewcommand{\d}{\delta}
\newcommand{\e}{\epsilon}
\newcommand{\f}{\phi}
\renewcommand{\l}{\lambda}
\renewcommand{\k}{\kappa}
\newcommand{\lhat}{\hat\lambda}
\newcommand{\m}{\mu}
\renewcommand{\o}{\omega}
\renewcommand{\r}{\rho}
\newcommand{\rbar}{{\bar\rho}}
\newcommand{\s}{\sigma}
\newcommand{\sbar}{{\bar\sigma}}
\renewcommand{\t}{\tau}
\newcommand{\z}{\zeta}
\newcommand{\bfmu}{\boldsymbol\mu}
\newcommand{\bfsigma}{\boldsymbol\sigma}
\newcommand{\bfxi}{\boldsymbol\xi}
\newcommand{\bfeta}{\boldsymbol\eta}
\newcommand{\bfzeta}{\boldsymbol\zeta}

\newcommand{\D}{\Delta}
\newcommand{\F}{\Phi}
\newcommand{\G}{\Gamma}

\newcommand{\gF}{{\mathfrak{F}}}
\newcommand{\gp}{{\mathfrak{p}}}
\newcommand{\gm}{{\mathfrak{m}}}
\newcommand{\gn}{{\mathfrak{n}}}
\newcommand{\gP}{{\mathfrak{P}}}
\newcommand{\gW}{{\mathfrak{W}}}

\def\Acal{{\mathcal A}}
\def\Bcal{{\mathcal B}}
\def\Ccal{{\mathcal C}}
\def\Dcal{{\mathcal D}}
\def\Ecal{{\mathcal E}}
\def\Fcal{{\mathcal F}}
\def\Gcal{{\mathcal G}}
\def\Hcal{{\mathcal H}}
\def\Ical{{\mathcal I}}
\def\Jcal{{\mathcal J}}
\def\Kcal{{\mathcal K}}
\def\Lcal{{\mathcal L}}
\def\Mcal{{\mathcal M}}
\def\Ncal{{\mathcal N}}
\def\Ocal{{\mathcal O}}
\def\Pcal{{\mathcal P}}
\def\Qcal{{\mathcal Q}}
\def\Rcal{{\mathcal R}}
\def\Scal{{\mathcal S}}
\def\Tcal{{\mathcal T}}
\def\Ucal{{\mathcal U}}
\def\Vcal{{\mathcal V}}
\def\Wcal{{\mathcal W}}
\def\Xcal{{\mathcal X}}
\def\Ycal{{\mathcal Y}}
\def\Zcal{{\mathcal Z}}

\renewcommand{\AA}{\mathbb{A}}
\newcommand{\BB}{\mathbb{B}}
\newcommand{\CC}{\mathbb{C}}
\newcommand{\FF}{\mathbb{F}}
\newcommand{\GG}{\mathbb{G}}
\newcommand{\NN}{\mathbb{N}}
\newcommand{\PP}{\mathbb{P}}
\newcommand{\QQ}{\mathbb{Q}}
\newcommand{\RR}{\mathbb{R}}
\newcommand{\ZZ}{\mathbb{Z}}

\def \bfa{{\boldsymbol a}}
\def \bfb{{\boldsymbol b}}
\def \bfc{{\boldsymbol c}}
\def \bfd{{\boldsymbol d}}
\def \bfe{{\boldsymbol e}}
\def \bff{{\boldsymbol f}}
\def \bfF{{\boldsymbol F}}
\def \bfg{{\boldsymbol g}}
\def \bfj{{\boldsymbol j}}
\def \bfp{{\boldsymbol p}}
\def \bfr{{\boldsymbol r}}
\def \bfs{{\boldsymbol s}}
\def \bft{{\boldsymbol t}}
\def \bfu{{\boldsymbol u}}
\def \bfv{{\boldsymbol v}}
\def \bfw{{\boldsymbol w}}
\def \bfx{{\boldsymbol x}}
\def \bfX{{\boldsymbol X}}
\def \bfy{{\boldsymbol y}}
\def \bfz{{\boldsymbol z}}

\newcommand{\abs}[1]{\left\lvert#1\right\rvert}
\newcommand{\Aut}{\operatorname{Aut}}
\newcommand{\can}{{\textup{can}}}
\newcommand{\Crit}{\operatorname{\mathcal{C}}}
\newcommand{\Branch}{\operatorname{\mathcal{B}}}
\newcommand{\Disc}{\operatorname{Disc}}
\newcommand{\Div}{\operatorname{Div}}
\newcommand{\drpt}{(\!(t)\!)}  
\newcommand{\dspt}{[\![t]\!]}  
\newcommand{\ddd}{D}  
\newcommand{\End}{\operatorname{End}}
\newcommand{\Fbar}{{\bar{\FF}}}
\newcommand{\Fix}{\operatorname{Fix}}
\newcommand{\Frob}{\operatorname{Frob}}
\newcommand{\Gal}{\operatorname{Gal}}
\newcommand{\hhat}{{\hat h}}
\newcommand{\Id}{\operatorname{Id}}
\newcommand{\Image}{\operatorname{Image}}
\newcommand{\isom}{\cong}
\newcommand{\Kbar}{{\bar K}}
\newcommand{\Mat}{\operatorname{Mat}}
\newcommand{\MOD}[1]{~(\textup{mod}~#1)}
\newcommand{\Mor}{\operatorname{Mor}}
\newcommand{\Moduli}{{\mathcal{M}}}
\newcommand{\Mult}{\operatorname{Mult}}
\newcommand{\Norm}{\operatorname{N}}
\newcommand{\notdivide}{\nmid}
\newcommand{\ord}{\operatorname{ord}}
\newcommand{\Pic}{\operatorname{Pic}}
\newcommand{\PCF}{\operatorname{PCF}}
\newcommand{\PCP}{\operatorname{PCP}}
\newcommand{\PGL}{\operatorname{PGL}}
\newcommand{\Per}{\operatorname{Per}}
\newcommand{\Qbar}{{\bar{\QQ}}}
\newcommand{\Rat}{\operatorname{Rat}}
\newcommand{\rank}{\operatorname{rank}}
\newcommand{\res}{\operatornamewithlimits{res}}
\newcommand{\Resultant}{\operatorname{Res}}
\renewcommand{\setminus}{\smallsetminus}
\def\SetMap(#1,#2){[#1:#2]}  
\newcommand{\Support}{\operatorname{Support}}
\newcommand{\Span}{\operatorname{Span}}
\newcommand{\Sym}{\operatorname{Sym}}
\newcommand{\Spec}{\operatorname{Spec}}
\newcommand{\sm}{{\textup{sm}}}
\newcommand{\tors}{{\textup{tors}}}
\newcommand{\bfzero}{\boldsymbol0}
\newcommand{\Zbar}{{\bar{\ZZ}}}
\newcommand{\<}{\langle}
\renewcommand{\>}{\rangle}
\newcommand{\bfP}{{\boldsymbol P}}
\newcommand{\bfC}{{\boldsymbol C}}
\newcommand{\lcm}{\operatornamewithlimits{lcm}}

\title[Post-Critically Finite Maps on $\mathbb{P}^n$]
{Post-Critically Finite Maps on $\mathbb{P}^n$ for $n\ge2$ are Sparse}
\date{\today}

\author{Patrick Ingram}
\address{Department of Mathematics and Statistics, York University,
  N520 Ross, 4700 Keele Street, Toronto, ON M3J 1P3, Canada}
\email{pingram@yorku.ca}

\author{Rohini Ramadas}
\address{Mathematics Department, Box 1917, Brown University,
Providence, RI 02912 USA}
\email{rohini\underline{ }ramadas@brown.edu}

\author{Joseph H. Silverman}
\address{Mathematics Department, Box 1917, Brown University,
  Providence, RI 02912 USA
  (ORCID: 0000-0003-3887-3248)}
\email{jhs@math.brown.edu}



\thanks{The first author's work was partially supported by Simons Collaboration
  Grant \#283120.
The second author's work was partially supported by NSF fellowship DMS-1703308.  
The third author's work was partially supported by Simons Collaboration
  Grant \#241309 and NSF EAGER DMS-1349908.}

\begin{abstract}
Let $f:{\mathbb P}^n\to{\mathbb P}^n$ be a morphism of degree
$d\ge2$. The map $f$ is said to be \emph{post-critically finite} (PCF)
if there exist integers $k\ge1$ and $\ell\ge0$ such that the critical
locus $\operatorname{Crit}_f$ satisfies
$f^{k+\ell}(\operatorname{Crit}_f)\subseteq{f^\ell(\operatorname{Crit}_f)}$. The
smallest such $\ell$ is called the \emph{tail-length}. We prove that
for $d\ge3$ and $n\ge2$, the set of PCF maps $f$ with tail-length at
most~$2$ is not Zariski dense in the the parameter space of all such
maps. In particular, maps with periodic critical loci, i.e., with
$\ell=0$, are not Zariski dense.
\end{abstract}

\maketitle

\tableofcontents

\newpage

\section{Introduction}
\label{section:introduction}

A rational map~$f:\PP^1\to\PP^1$ of degree~$d\ge2$ is said to be
\emph{post-critically finite} (PCF) if all of its critical points have
finite forward orbits.  PCF maps play a fundamental role in the study
of one-dimensional dynamics; see Remark~\ref{remark:pcf1} for a brief
history. In particular, PCF maps are ubiquitous in the sense that they are
Zariski dense in the parameter space of all degree~$d$ rational maps of~$\PP^1$,
and the same is true of the smaller collection of post-critically
periodic (PCP) maps, which are the maps whose critical points are
periodic; see~\cite[Theorem~A]{MR3831028}.

Forn{\ae}ss and Sibony~\cite{MR1190986} introduced an analogue of PCF
maps on~$\PP^n$ for~$n\ge2$, and a number of authors have constructed
examples of such maps and studied their properties;
see~\cite{MR2581826,MR1940161,MR2603592,MR3492630,MR1609475,MR3107522,MR2415046,MR1747345}
for examples in complex dynamics, and~\cite{MR3762687,MR3492630} for
some arithmetic results.  Our aim in this paper is to explain why it
is likely that the set of such maps is much sparser than in the
one-dimensional case, and to prove a result which quantifies this
statement for~PCF maps having small tail length.  We set the notation
\[
  \End_d^n := \left\{
    \begin{tabular}{@{}c@{}}
      morphisms $f:\PP^n\to\PP^n$ of algebraic\\
      degree $d$, i.e., $f^*\Ocal_\PP(1)=\Ocal_\PP(d)$\\
    \end{tabular}
  \right\}. 
\]
We note that~$\End_d^n$ is naturally identified with a Zariski open
subset of~$\PP^N$, where~$N=(n+1)\binom{d+n}{n}-1$. More precisely,
the variety~$\End_d^n$ is the complement of the hypersurface
in~$\PP^N$ defined by the vanishing of the Macaulay
resultant. See~\cite[Chapter~1]{MR2884382} for details.

In this paper we always work over\footnote{Some parts of this paper
  remain true over infinite fields of characteristic~$p$, but to avoid
  separability complications, we restrict to the case of
  characteristic~$0$.}
\[
\FF := \text{an algebraically closed field of characteristic~$0$.}
\]

\begin{definition}
The \emph{critical locus} of a map~$f=[f_0,\ldots,f_n]\in\End_d^n$
given by homogeneous polynomials~$f_i(x_0,\ldots,x_n)$ is the variety
\[
  \Crit_f := \left\{ \det\left( \frac{\partial f_i}{\partial x_j}\right) = 0 \right\}
  \subset \PP^n.
\]
The \emph{branch locus} of~$f$ is the image of the critical locus, taken with
the reduced scheme structure and denoted by
\[
  \Branch_f := f(\Crit_f).
\]
\end{definition}

\begin{definition}
\label{definition:PCF}
A map~$f\in\End_d^n$ is \emph{post-critically finite} (PCF) if
there exist~$k\ge1$ and~$\ell\ge0$ such that
\[
  f^{k+\ell}(\Crit_f) \subseteq f^\ell(\Crit_f).
\]
If~$k$ and~$\ell$ are chosen minimally, we say that~$f$ is \emph{PCF
  of Type~$(k,\ell)$}, where~$k$ is the \emph{period}
and~$\ell$ is the \emph{tail-length}. A PCF map with tail length~$0$ is
said to be \emph{post-critically periodic} (PCP).
\end{definition}

Our main theorem says that in dimension greater than one,
post-critically periodic maps are comparatively rare, and more
generally the same is true for post-critically finite maps whose
tail-length is at most~$2$.

\begin{theorem}
\label{theorem:PCPnotdenseinRatdn}
Let~$d\ge3$ and~$n\ge2$. Fix some $\ell\le2$. Then
\[
  \{ f\in\End_d^n : \text{$f$ is post-critically finite of Type $(k,\ell)$ for some $k\in\NN$} \}
\]
is contained in a proper Zariski closed subset of~$\End_d^n$.
\end{theorem}

We conjecture that Theorem~\ref{theorem:PCPnotdenseinRatdn} is true
for any fixed tail-length, and we ask whether it remains valid for the union
over all tail-lengths.

\begin{conjecture}
\label{conjecture:PCPnotdenseinRatdn}
Let~$d\ge3$ and~$n\ge2$. Then for all~$\ell\ge1$,
\[
  \{ f\in\End_d^n : \text{$f$ is post-critically finite of Type $(k,\ell)$ for some $k\in\NN$} \}
\]
is contained in a proper Zariski closed subset of~$\End_d^n$.
\end{conjecture}

\begin{question}
\label{question:PCPnotdenseinRatdn}
Let~$d\ge3$ and~$n\ge2$.  Is the set
\[
  \{ f\in\End_d^n : \text{$f$ is post-critically finite} \}
\]
contained in a proper Zariski closed subset of~$\End_d^n$?
\end{question}

\begin{remark}
  \label{remark:pcf1}
One motivation for studying PCF endomorphisms in higher dimensions
comes from work of Nekrashevych~\cite{MR3199801}, in which he studies
the Julia set of a PCF map $f:\PP^N_\CC\to \PP^N_\CC$ using an
associated iterated monodromy group. In~\cite{MR2581826}, Belk and
Koch explicitly compute the iterated monodromy group associated to a
particular example, which in fact turns out to be post-critically
periodic. We also mention that the algebraic analogue of the partial
self-covering property is exploited in~\cite{MR3762687} to show that
extensions of number fields obtained by adjoining backward orbits of
points relative to PCF endomorphisms of any smooth, projective variety
are finitely ramified.
\end{remark}

\begin{remark}
For ease of exposition, we work in the parameter space $\End_d^n$,
but we note that since the PCF property is invariant under
$\PGL_{n+1}$-con\-ju\-ga\-tion,
Theorem~\ref{theorem:PCPnotdenseinRatdn} could equally well be
formulated for the dynamical moduli
space~$\Moduli_d^n:=\End_d^n/\!/\PGL_{n+1}$ constructed via GIT
in~\cite{MR2741188,MR2567424}. And similarly for
Conjecture~\ref{conjecture:PCPnotdenseinRatdn} and
Question~\ref{question:PCPnotdenseinRatdn}.
\end{remark}

\begin{remark}
The property of being PCF as given in Definition~\ref{definition:PCF}
admits two other equivalent characterizations that are sometimes
useful. First, a map $f\in \End_d^n$ is PCF if and only if the
post-critical locus
\[
\operatorname{PostCrit}(f) := \bigcup_{m\ge 1} f^m(\Crit_f)
\]
is algebraic, that is, if~$\operatorname{PostCrit}(f)$ consists of a
finite union of algebraic hypersurfaces. This equivalence follows
immediately from the fact that for each~$m$, the image $f^m(\Crit_f)$ is a finite union of
algebraic hypersurfaces. Second, a map $f$ is PCF if and
only if there exists a Zariski-open subset $U\subseteq \PP^N$ such that
$f^{-1}(U)\subseteq U$ and such that $f:U\to \PP^N$ is unramified;
specifically, if such a $U$ exists, then its complement is algebraic
and contains the post-critical locus.
\end{remark}

We briefly summarize the contents of this paper. In
Section~\ref{section:examples} we give various constructions of~PCF
maps and non-PCF maps, and in particular show that for all~$d$
and~$n$, every period and tail length can occur.
In Section~\ref{section:detvars} we prove that there is a
Zariski dense set of~$f\in\End_d^n$ such that~$\Crit_f$ is a variety
of general type.  (We thank Jason Starr for showing us this proof.)
We use this in Section~\ref{section:pbpmapssparse} to show that the
set of~PCP maps, i.e., the set of maps~$f$ of PCF Type~$(k,0)$, is not
Zariski dense. Section~\ref{section:twomultiplicitylemmas} contains
two multiplicity lemmas. In
Section~\ref{section:hyperplaneconstruction} we construct maps whose
branch locus has a minimally branched point and use this map to show
that the set of~$f$ of PCF Type~$(k,1)$ is not Zariski
dense. Section~\ref{section:fixedtaillength} gives a general method
for proving, for any fixed~$\ell$, that the set of~$f$ of PCF
Type~$(k,\ell)$ is not Zariski dense. This method requires showing
that there exists a single map having certain properties. In
Section~\ref{section:taillength2} we construct such a map
for~$\ell=2$, thereby completing the proof that the set of~$f$ of PCF
Type~$(k,2)$ is not Zariski dense.

\section{Examples of PCF maps}
\label{section:examples}

Before proving our main results on higher dimensioal PCF maps, we
pause in this section to give a number of examples. We remark that in
all of these examples, the critical locus~$\Crit_f$ is reducible, and
indeed it is generally a union of rational hypersurfaces, the
multiplicity~$\Mult_{\Crit_f}(f)$ is strictly greater than~$2$ and
generally equal to~$\deg(f)$, and the restriction
$f|_{\Crit_f}:\Crit_f\to\Branch_f$ is generally not~$1$-to-$1$.  This
highlights the difficulty of constructing maps whose critical and
branch loci are sufficiently generic, and the existence of such maps
is the key to proving results such as
Theorem~\ref{theorem:PCPnotdenseinRatdn}.

\begin{example}\label{ex:dthpower}
The most obvious PCF map is the $d$-power map
\[
\text{$f=[x_0^d,\ldots,x_n^d]$ with critical locus $\Crit_f=\{x_0x_1\cdots{x_n}=0\}$}
\]
consisting of the coordinate hyperplanes. Thus $f(\Crit_f)=\Crit_f$,
so~$f$ is~PCF of Type~$(1,0)$.
\end{example}

\begin{example}[Symmetric powers of PCF maps on $\PP^1$]  
\label{example:sympowerpcf}  
Let $f:\PP^1\to\PP^1$ be a PCF map of degree $d$. Then the $n$-fold product map,
which we denote by
\[
F_n := f\times f\times\cdots\times f :(\PP^1)^n\longrightarrow(\PP^1)^n,
\]
descends to a map on the symmetric prodcut $(\PP^1)^n/\Scal_n$. Using
the standard isomorphism $\PP^n\cong(\PP^1)^n/\Scal_n$, we obtain a
map~${\tilde{F}_n}$ on~$\PP^n$ such that the diagram in
Figure~\ref{figure:SymmetricPower} commutes.
\begin{figure}[ht]
\[
\begin{CD}
(\PP^1)^n @>F_n>> (\PP^1)^n \\
@V \pi VV
@VV \pi V \\
\PP^n @>{\tilde{F}_n}>> \PP^n \\
\end{CD}
\]
\caption{The symmetric power of a PCF map}
\label{figure:SymmetricPower}
\end{figure}

The vertical quotient map, denoted by~$\pi$ in
Figure~\ref{figure:SymmetricPower}, is $n!$-to-$1$, and its critical
locus is the big diagonal in~$(\PP^1)^n$, i.e.,
\[
\Crit_\pi = \bigl\{ (p_1,\ldots,p_n)\in(\PP^1)^n : \text{$p_i=p_j$ for some $i\ne j$} \bigr\}.
\]
We denote the branch of~$\pi$ by
\[
\Branch_\pi := \pi(\Crit_\pi) \subset \PP^n.
\]
We observe that $\Crit_\pi$ is reducible and that its irreducible
components are rational, but all of the components have the same image
in~$\PP^n$.  Thus $\Branch_\pi$ is an irreducible rational
hypersurface in~$\PP^n$.  Further, since the topological degree of
${\tilde{F}_n}$ is equal to the topological degree of~$F_n$, which
is~$d^n$, we see that the algebraic degree of~${\tilde{F}_n}$ is~$d$.

The postcritical portrait of a PCF map is the self-map of finite sets
induced on the set of irreducible components of the postcritical
locus. We now describe how the postcritical portrait
of~${\tilde{F}_n}$ can be deduced from the postcritical portrait
of~$f$.  We are assuming that every~$p\in\Crit_f$ is pre-periodic
under~$f$, and we denote the tail-length and period of~$p$ by~$\ell_p$
and~$k_p$, respectively.

The commutative diagram in Figure~\ref{figure:SymmetricPower} and the
chain rule give
\[
\Crit_{{\tilde{F}_n}}\subseteq \pi(\Crit_{F_n}\cup(F_n)^{-1}(\Crit_{\pi}))
\quad\text{and}\quad
\pi((\Crit_{F_n}\cup(F_n)^{-1}(\Crit_{\pi}))\setminus \Crit_{\pi})\subseteq \Crit_{{\tilde{F}_n}}.
\]
For~$p\in\PP^1$ we denote by~$H_p$ the reducible hypersurface
\[
\bigl\{(p_1,\ldots p_n) : \text{$p_i=p$ for some $i$} \bigr\} \subset (\PP^1)^n,
\]
and we denote by ${\tilde{H}}_p:=\pi(H_p)$ its irreducible, rational
image in~$\PP^n$. We have that
\[
\Crit_{F_n}=\bigcup_{p\in\Crit_f} H_p.
\]
Since $H_p$ is not contained in $\Crit_{\pi}$, we
conclude that ${\tilde{H}}_p\subseteq\Crit_{{\tilde{F}_n}}$, so
$\pi(\Crit_{F_n})\subseteq \Crit_{{\tilde{F}_n}}$. We denote by
$E$ the reducible hypersurface
\[
E := 
\bigcup_{i\ne j} \bigl\{(p_1,\ldots p_n) : \text{$p_i\ne p_j$ and $f(p_i)=f(p_j)$}\bigr\}\subseteq (\PP^1)^n,
\]
and we set $\tilde{E}:=\pi(E)\subseteq\PP^n$.
We have
\[
F_n(\Crit_\pi)=\Crit_\pi \quad\text{and}\quad
(F_n)^{-1}(\Crit_{\pi}) = \Crit_\pi \cup E,
\]
which implies that
\[
\tilde{E}\subseteq \Crit_{{\tilde{F}_n}},\quad
{\tilde{F}_n}(\tilde{E})=\Branch_\pi,\quad\text{and}\quad
{\tilde{F}_n}(\Branch_\pi)=\Branch_\pi.
\]
Thus $\tilde{E}$ is pre-periodic with tail-length and period both equal to~$1$.

We know that $\Crit_\pi\not\subseteq\Crit_{F_n}$, so by comparing the
multiplicities along $\Crit_\pi$ of $\pi\circ F_n$ and
${\tilde{F}_n}\circ\pi$, we see that
$\Branch_\pi=\pi(\Crit_\pi)\not\subseteq\Crit_{{\tilde{F}_n}}$. This
lets us conclude that
\[
\Crit_{{\tilde{F}_n}}= \tilde{E} \cup \bigcup_{p\in\Crit_f} {\tilde{H}}_p.
\]

For all $p\in\PP^1$ we have $F_n(H_p)=H_{f(p)}$, and thus
${\tilde{F}_n}({\tilde{H}}_p)={\tilde{H}}_{f(p)}$. Hence the irreducible
component ${\tilde{H}}_p$ of $\Crit_{{\tilde{F}_n}}$ is pre-periodic with
tail-length $\ell_p$ and period $k_p$. Putting this all together, we
deduce that for all $n\ge2$, the map~${\tilde{F}_n}$ is PCF of Type
$(k,\ell)$ with
\begin{equation}
  \label{eqn:kellfsymn}
  k=\lcm_{p\in\Crit_f}k_p
  \quad\text{and}\quad
  \ell=\max\bigl( 1, \max_{p\in\Crit_f}\{\ell_p\} \bigr).
\end{equation}
\end{example}

As a special case of Example~\ref{example:sympowerpcf}, we obtain the
following result.

\begin{proposition}
For all~$n\ge1$, all~$d\ge2$, all~$k\ge1$ and all~$\ell\ge1$, there
exists a PCF map of Type~$(k,\ell)$ in~$\End_d^n$.
\end{proposition}
\begin{proof}
It is known that for all~$d\ge2$, all~$k\ge1$ and all~$\ell\ge1$,
there exists a PCF map~$f$ of Type~$(k,\ell)$ in~$\End_d^1$
such that~$f$ has exactly two critical points, one
fixed, and one pre-periodic of Type~$(k,\ell)$. More precisely, one
can take~$f(x)=x^d+c$ for an appropriate choice of~$c$.
It follows from~\eqref{eqn:kellfsymn} that~${\tilde{F}_n}\in\End_d^n$
is PCF of Type~$(k,\ell)$.
\end{proof}

\begin{example}
Koch~\cite{MR3107522} has used Teichm\"uller theory and Thurston's
topological characterization of PCF maps on~$\PP^1$~\cite{MR1251582}
to construct interesting PCF maps in all dimensions and degrees.  We
give a brief overview.

Let~$\phi:S^2 \to S^2$ be a degree~$d$ orientation preserving
branched covering from a topological~$2$-sphere to itself. Suppose
further that~$\phi$ is PCF, i.e., has finite post-critical set
\[
\bfP:=\bigl\{\phi^n(x) : \text{ $x$ is a critical point of $\phi$ and $n>0$}\bigr\}.
\]
Then that there is a holomorphic pullback map
\[
\sigma_\phi : \Tcal(S^2,\bfP)\longrightarrow\Tcal(S^2,\bfP)
\]
induced by~$\phi$ on the Teichm\"uller space~$\mathcal{T}(S^2,\bfP)$
of complex structures on~$(S^2,\bfP)$, and the following are
equivalent:
\begin{parts}
  \Part{\textbullet}
  The branched covering~$\phi$ is homotopic to a PCF rational function
  on~$\PP^1(\CC)$.
  \Part{\textbullet}
  $\sigma_{\phi}$ has a fixed point.
\end{parts}
See~\cite{MR1251582} for details of this result, which is due to
Thurston. Teichm\"uller space $\mathcal{T}(S^2,\bfP)$ is a
non-algebraic complex manifold, but it is the universal cover of the
algebraic moduli space~$\Moduli_{0,\bfP}$ of markings
of~$\PP^1(\CC)$ by the set~$\bfP$. In turn,
$\Moduli_{0,\bfP}$ can be identified with a Zariski open subset of
$\PP^{\abs{\bfP}-3}$. The complement of~$\Moduli_{0,\bfP}$ in
$\PP^{\abs{\bfP}-3}$ is a union of hyperplanes:
\[
\Delta:=
\bigcup_{0\le i\le\abs{\bfP}-2} \{x_i=0\}
\cup \bigcup_{0\le i<j\le\abs{\bfP}-2} \{x_i=x_j\}.
\]

In~\cite{MR3107522}, Koch introduced PCF endomorphisms of
$\PP^{\abs{\bfP}-3}$ that descend from the transcendental Thurston
pullback map. Suppose~$\phi:(S^2,\bfP)\to(S^2,\bfP)$ is PCF and
satisfies:
\begin{enumerate}
\item $\bfP$ contains a totally ramified fixed point~$p_\infty$ of
  $\phi$, i.e.,~$\phi$ is a \textit{topological polynomial}.
\item Either:
\begin{enumerate}
\item Every other critical point of~$\phi$ is also periodic, i.e.,~$\phi$ is PCP.
\item There is exactly one other critical point,~$p_0$, of~$\phi$,
  so in particular~$p_0$ is pre-periodic. \label{it:PCFbranched}
\end{enumerate}
\end{enumerate}
Then the inverse of~$\sigma_{\phi}$ descends to~$\PP^{\abs{\bfP}-3}$;
i.e., there is a degree~$d$ PCF map
$R_{\phi}:\PP^{\abs{\bfP}-3}\to\PP^{\abs{\bfP}-3}$ such that the
diagram in Figure~\ref{figure:TS2P} commutes:

\begin{figure}[ht]
\[
\begin{CD}
\Tcal(S^2,\bfP) @>\sigma_\phi>> \Tcal(S^2,\bfP) \\
@V \substack{\text{universal}\\\text{cover}\\} VV
@VV \substack{\text{universal}\\\text{cover}\\} V \\
\Mcal_{0,\bfP} @. \Mcal_{0,\bfP} \\
@V \substack{\text{open}\\\text{inclusion}\\} VV
@VV \substack{\text{open}\\\text{inclusion}\\} V \\
\PP^{\abs{\bfP}-3} @< R_\phi << \PP^{\abs{\bfP}-3}\\
\end{CD}
\]
\caption{A PCF map via Teichm\"uller theory}
\label{figure:TS2P}
\end{figure}

Koch describes the critical locus of~$R_{\phi}$ as well as its
postcritical portrait. The critical locus of
$R_{\phi}$ is contained in~$\Delta$, and thus it is a union of
hyperplanes. In particular,~$R_\phi$ is reducible, and no component is
of general type. The postcritical locus of~$R_{\phi}$ is equal
to~$\Delta$, and thus is a union of
exactly~$N:=(\abs{\bfP}-2+\binom{\abs{\bfP}-2}{2})$ hyperplanes. Hence
there are upper bounds on the possible period~$k$ and
tail-length~$\ell$ of the PCF map~$R_{\phi}$, in terms of the
dimension~$n=\abs{\bfP}-3$ of the projective space.

The postcritical portrait of~$R_{\phi}$ can be deduced from the
postcritical portrait of~$\phi$; for a complete description
see~\cite[Propositions~6.1 and~6.2]{MR3107522}. In particular,
$R_{\phi}$ is PCP if and only if~$\phi$ is PCP, and in this case for
every periodic critical point~$p\ne p_{\infty}$ of~$\phi$, there is at
least one critical hyperplane of~$R_{\phi}$ in a periodic cycle of the
same length as the cycle of~$p$. It follows from the proof
of~\cite[Proposition 6.1]{MR3107522} that if~$R_\f:\PP^n\to\PP^n$ is
PCF, then the components of its critical locus have period at
most~$\lceil{n}^2/4+n\rceil$.

Next suppose that~$\phi$ is not PCP, so we assume that it satisfies
Condition~\ref{it:PCFbranched} given earlier, i.e., there are exactly
two critical points~$p_0$ and~$p_{\infty}$ of~$\phi$, with~$p_\infty$
fixed, and~$p_0$ pre-periodic. In this case, $R_{\phi}$ can be written
as $\alpha\circ f$, where $f$ is the $d$-th power map described in
Example \ref{ex:dthpower}, and $\alpha$ is some automorphism of
$\PP^n$. Thus $\Crit_{R_{\phi}}$ is the union of the coordinate
hyperplanes $H_i:=\{x_i=0\}$ for $0\le{i}\le{n}$. Each $H_i$
eventually maps into a periodic cycle of hyperplanes in $\Delta$, but
no $H_i$ is itself periodic. It follows from counting the number of
hyperplanes in $\Delta$ that if $R_{\phi}$ is PCF Type~$(k,\ell)$ then
$1\le k,\ell\le \binom{n+1}{2}$.
\end{example}

\begin{example}
Although our contention is that PCF maps are rare, it is perhaps not
obvious that there exist any maps that are not PCF. So we take the
time here to construct examples of non-PCF maps in~$\End_d^n$ for
all~$d\ge2$ and all~$n\ge1$.  We consider the family of morphisms
\[
f_t:\PP^n\longrightarrow\PP^n,\quad
f_t(X_0,\ldots,X_n) := [X_0^d-tX_1^d,X_1^d,X_2^d,\ldots,X_n^d].
\]
The critical locus of~$f_t$ is
\[
\Crit_{f_t} = \{ X_0X_1\cdots X_n = 0 \}.
\]
In particular, the hyperplane
\[
H_0 := \{ X_0=0 \} \subset \PP^n
\]
is a component of~$\Crit_{f_t}$. We claim that for all values of~$t$ and all~$k\ge0$,
the image~$f^k(H_0)$ is a hyperplane, and that
for most~$t$ values, the sequence of hyperplanes~$f^k(H_0)$ does not repeat.
To see this,  we define a new map
\[
\f : \PP^1\longrightarrow\PP^1,\quad
\f(t) = t^d+t.
\]
An easy induction shows that
\[
f^k(H_0) = \bigl\{ X_0 = \f^{k-1}(t)\cdot X_1 \bigr\} \subset \PP^n
\quad\text{for all $k\ge1$.}
\]
Thus~$f^k(H_0)$ is a hyperplane for all~$k\ge0$, and if we choose
any~$t\in\FF$ that is not preperiodic for~$\f$, then the
hyperplanes~$f^k(H_0)$ will be distinct.  So we are reduced to showing
that the map~$\f$ has non-preperiodic points in the algebraically
closed characteristic~$0$ field~$\FF$.  If~$\FF=\CC$, this is obvious,
since the set of preperiodic points is countable. For a countable
field such as~$\Qbar$, one can use the fact that the preperiodic
points are a set of bounded height, so there are non-preperiodic
points for~$\f$ in~$\Qbar$, and indeed in~$\ZZ$.  We leave the general case
to the reader.

We observe as an immediate consequence that for any fixed~$k$
and~$\ell$, the set
\begin{equation}
  \label{eqn:extykl}
  \{ f\in\End_d^n : \text{$f$ is post-critically finite of type $(k,\ell)$} \}
\end{equation}
is not Zariski dense in~$\End_d^n$. This follows, since elimination
theory says that the set~\eqref{eqn:extykl} is Zariski closed, and our
example says that the complement of~\eqref{eqn:extykl} is non-empty,
so~\eqref{eqn:extykl} is not all of~$\End_d^n$. Of course, the fact
that~\eqref{eqn:extykl} is not Zariski dense for each fixed
pair~$(k,\ell)$ is much weaker than
Theorem~\ref{theorem:PCPnotdenseinRatdn}, which implies that
if~$\ell\le2$, then the set~\eqref{eqn:extykl} is empty for all but
finitely many~$k$.  And an affirmative answer to
Question~\ref{question:PCPnotdenseinRatdn} would be even stronger,
since it would say that~\eqref{eqn:extykl} is empty for all but
finitely many~$(k,\ell)$ pairs.
\end{example}


\section{Determinental varieties are of general type}
\label{section:detvars}
A key tool in the proof that PCP maps are sparse is the following
result, whose proof was shown to us by Jason Starr.

\begin{theorem}
\label{theorem:critgentype}
Let~$n\ge2$ and~$d\ge3$. Then the set
\[
  \{ f\in\End_d^n : \text{$\Crit_f$ is an irreducible variety of general type} \}
\]
is a non-empty Zariski open subset of~$\End_d^n$.
\end{theorem}
\begin{proof}
The generic determinantal variety
\[
  \Dcal 
   = \bigl\{M\in\Mat_{(n+1)\times (n+1)}(\FF)\cong\FF^{(n+1)^2} : \det(M)=0 \bigr\}
\]
is singular, but its singularities are relatively mild.
More precisely, the generic determinantal variety~$\Dcal$ is
canonical, and thus all global sections of (positive powers of) the
dualizing sheaf on the singular determinantal variety lift to global
sections (rather than to rational/mero\-morphic sections) of (positive
powers of) the dualizing sheaf on any desingularization.  This follows
from results of Vainsencher~\cite{MR738261}, who describes an explicit
desingularization~$\tilde\Dcal$ of~$\Dcal$ as the space of complete
linear collineations. The result is also stated explicitly and proven in 
the preprint of Starr~\cite[Corollary~3.14]{arxiv0305432}.
\par
Since~$\Dcal$ has canonical singularities, and since the total space
of the incidence correspondence is smooth over the parameter space of
matrices, it follows that the inverse image of~$\Dcal$ in this total
space also has canonical singularities. Thus when we project this
total space to the parameter space of~$(n+1)$-tuples of homogeneous
degree~$d$ polyomials, the (geometric) generic fiber has canonical
singularities. Hence the open set of the parameter space consisting
of fibers that have canonical singularities is dense. 
\par
Since these fibers have canonical singularities, they are of general
type once the dualizing sheaf is ample. But for degree~$d$ maps
$\PP^n\to\PP^n$, the dualizing sheaf of the critical locus is the
restriction of
\[
  \Ocal_{\PP^n}\bigl((n+1)(d-2)\bigr).
\]
Hence if~$d\ge3$, then a general~$(n+1)$-tuple of degree~$d$
homogeneous polynomials has a critical locus whose desingularization
is of general type.
\end{proof}

Theorem~\ref{theorem:critgentype} covers maps of degree~$d\ge3$ for
all dimensions~$n$.  For dimension~$2$ we can prove something stronger
that includes quadratic maps.

\def\EEE{\Ecal_d^{\textup{sm-irr}}}

\begin{theorem}
\label{theorem:A2case}
We consider the set of maps
\[
  \EEE := 
  \{ f\in\End_d^2 : \text{$\Crit_f$ is smooth and irreducible} \}.
\]
\vspace{-10pt}
\begin{parts}
  \Part{(a)}
  Let $d\ge1$. Then~$\EEE$ is a non-empty
  Zariski open subset of~$\End_d^2$.
  \Part{(b)}
  Let $d\ge2$. Then~$\EEE$ does not contain any PCP maps.
  \Part{(c)}
  For all $d\ge2$, the set
  \[
  \{f\in\End_d^n:\text{$f$~is~PCP}\}
  \]
  is not Zariski dense in~$\End_d^n$.
\end{parts}
\end{theorem}
\begin{proof}
(a)\enspace  
The set~$\EEE$ is clearly Zariski open, so the only
question is whether it's empty.  To prove that~$\EEE$
is not empty, we use~\cite[Theorem~1]{MR2775814}, which says that for
any smooth irreducible surface~$S\subset\PP^r$, the set of linear
projections $\pi:\PP^r\to\PP^2$ such that the critical locus
of~$\pi|_S$ is smooth and irreducible is a non-empty Zariski open
subset of the space of linear projections. (The special case that~$S$
is a Veronese embedding of~$\PP^2$ is proven in~\cite{MR1817250}.)
Taking~$S$ to be the image of the~$d$-uple
embedding~$\rho_d:\PP^2\hookrightarrow\PP^r$~\cite[Exercise~I.2.12]{hartshorne},
we see that compositions with linear projections~$\pi\circ\rho_d$
correspond exactly to degree~$d$ rational maps~$\PP^2\to\PP^2$. So the
desired result is the special case of~\cite{MR2775814} in
which~$S=\rho_d(\PP^2)$.
\par\noindent(b)\enspace
Let~$f\in\EEE$.  Then~$\Crit_f$ is a
smooth irreducible curve of degree~$3(d-1)$ in~$\PP^2$, so it has
genus $g(\Crit_f)=\frac12(3d-4)(3d-5)\ge1$ for
all~$d\ge2$.\footnote{For~$d\ge3$, the genus
  satisfies~$g(\Crit_f)\ge10$, so in particular~$\Crit_f$ is of
  general type, as predicted by Theorem~\ref{theorem:critgentype}; but
  for~$d=2$ we see that~$\Crit_f$ is not of general type. This shows
  that Theorem~\ref{theorem:critgentype} cannot be extended to~$d=2$.}
Suppose now that~$f$ is~PCP, so~$f^k(\Crit_f)=\Crit_f$ for
some~$k\ge1$.  (Note that we must have equality, since~$\Crit_f$ is
irreducible.)  Thus~$\Crit_f$ is an irreducible curve that is
(forward) invariant for the map~$f^k$.
Further, since
\[
\Crit_{f^k}=\Crit_f+f^*\Crit_f+\cdots +f^{(k-1)*}\Crit_f,
\]
we see that~$\Crit_f$ is also critical for~$f^k$. We now
apply~\cite[Theorem~4.1]{MR1940161}, which says that an irreducible
curve in~$\PP^2$ that is forward invariant for a non-linear
morphism~$\PP^2\to\PP^2$ is necessarily a rational curve, i.e., has
genus~$0$. This contradicts~$g(\Crit_f)\ge1$, which completes the
proof that the set~$\EEE$ does not contain any PCP
maps.
\par\noindent(c)\enspace
This is immediate from~(a) and~(b), since~(a) gives a non-empty
Zariski open subset of~$\End_d^2$, and~(b) says that this open set
contains no PCP maps.
\end{proof}

\section{Proof that post-critically periodic maps are sparse}
\label{section:pbpmapssparse}

In this section we prove the tail length~$0$ part of
Theorem~\ref{theorem:PCPnotdenseinRatdn}, i.e., we prove the following
result:

\begin{theorem}
\label{theorem:PCPfinite}
Let~$d\ge3$ and~$n\ge2$. Then
\[
  \{ f\in\End_d^n : \text{$f$ is post-critically periodic} \}
\]
is contained in a proper Zariski closed subset of~$\End_d^n$.
\end{theorem}

\begin{proof}
For notational convenience we let
\[
  \PCP_d^n :=   \{ f\in\End_d^n : \text{$f$ is post-critically periodic} \}.
\]
We assume that~$\PCP_d^n(\FF)$ is a Zariski dense subset  
of~$\End_d^n(\FF)$ and derive a contradiction. 

\Step
Theorem~\ref{theorem:critgentype} tells us that
\begin{equation}
  \label{eqn:EnddnCfirrgt}
  \{f\in\End_d^n(\FF) : \text{$\Crit_f$ is
         irreducible and of general type} \}
\end{equation} 
is a non-empty Zariski open subset of~$\End_d^n(\FF)$.  Under our
assumption that~$\PCP_d^n(\FF)$ is a Zariski dense subset
of~$\End_d^n(\FF)$, it follows that the intersection
of~$\PCP_d^n(\FF)$ with~\eqref{eqn:EnddnCfirrgt}, i.e., the set
\[
  \{f\in\PCP_d^n(\FF) : \text{$\Crit_f$ is
         irreducible and of general type} \},
\] 
is also a Zariski dense subset of~$\End_d^n(\FF)$.

\Step
We next show that for every map~$f$ in the set
\[
  \{f\in\PCP_d^n(\FF) : \text{$\Crit_f$ is
         irreducible and of general type} \},
\]
there is an integer~$m(f)\ge1$ such that
\[
\Crit_f \subseteq \Fix(f^{m(f)}),
\]
i.e., there is an iterate of~$f$ that fixes every point in~$\Crit_f$.
To see this, we use the definition of~PCP to find some~$k\ge1$ such
that~$f^k(\Crit_f)\subseteq\Crit_f$. But~$f$ is a morphism, so for any
irreducible subvariety~$V\subset\PP^n$ we
have~$\dim{f}(V)=\dim{V}$. Hence~$\dim f^k(\Crit_f)=\dim\Crit_f$, and
the irreducibility of~$\Crit_f$ implies
that~$f^k(\Crit_f)=\Crit_f$. In other words, the map~$f^k|_{\Crit_f}$
is a surjective endomorphism of~$\Crit_f$. But~$\Crit_f$ is of general
type, and it is known that for varieties of general type, every
surjective endomorphism is an automorphism;
see~\cite[Lemma~3.4]{MR2509692} or~\cite[Proposition~10.10]{MR637060}.
Further, the automorphism group of a variety of general type is
finite; see~\cite{MR3034294} for a recent strong upper bound on its
order.\footnote{The quintessential example is that of a curve of
  genus~$g\ge2$, whose automorphism group has order at
  most~$84(g-1)$.}  Hence there exists an~$r$ such that~$f^{kr}$ fixes
every point of~$\Crit_f$, and we take~$m(f)=kr$.

\Step
We prove two useful lemmas.

\begin{lemma}
\label{lemma:bdhtset}
Let~$V\subset\PP^n$ be an irreducible projective variety defined over~$\Qbar$
such that~$V(\Qbar)$ is a set of bounded height. Then~$\dim(V)=0$.
\end{lemma}
\begin{proof}
Let~$K/\QQ$ be a number field that is a field of definition for~$V$,
and let~$D:=\dim(V)$. Taking a projection onto a generic
dimension~$D$ linear subspace of~$\PP^n$ defined over~$K$ gives a
generically finite dominant rational map~$\pi:V\dashrightarrow\PP^D$.
We take non-empty Zariski open subsets~$V^\circ\subseteq{V}$
and~$U^\circ\subseteq{\PP^D}$ defined over~$K$ so that~$\pi^\circ:V^\circ\to{U^\circ}$
is a finite morphism, and we let~$r=\deg(\pi^\circ)$.
Then the map
\begin{equation}
  \label{eqn:picircLKrVL}
\pi^\circ : \bigcup_{[L:K]\le r} V^\circ(L) \longmapsto U^\circ(K)
\end{equation}
is surjective, since every point in~$U$ has at most~$r$ points in its
inverse image. The points in~$V^\circ(L)$ have bounded height, since
we have assumed that~$V(\Qbar)$ is a set of bounded height, and they
also clearly lie in a field of bounded degree over~$K$.  In other
words, the set on the left-hand side of~\eqref{eqn:picircLKrVL} is a
set of bounded degree and bounded height, so it is a finite set;
see~\cite[Theorem~3.7]{MR2316407}.  The surjectivity
of~\eqref{eqn:picircLKrVL} implies that~$U^\circ(K)$ is finite.
But~$U^\circ$ is a non-empty Zariski open subset of~$\PP^D$ that is
defined over~$K$, so~$U^\circ(K)$ finite implies that~$D=0$.
\end{proof}

\begin{lemma}
\label{lemma:fixdim0}
Let~$f\in\End_d^n(\FF)$ with~$d\ge2$. Then~$\dim\Fix(f)=0$.
\end{lemma}
\begin{proof}
We first note that
\[
\Fix(f) = \bigcap_{0\le i<j\le n} \{x_jf_i - x_if_j = 0\}
\]
is an intersection of~$\frac12(n^2+n)$ hypersurfaces of degree~$d+1$.
So~$\Fix(f)$ is a subvariety of~$\PP^n$, and by a weak form of
Bezout's theorem, we know that either
\[
  \dim\Fix(f)\ge1
  \quad\text{or}\quad
  \#\Fix(f) \le D(n,d),
\]
where as indicated~$D(n,d)$ depends only on~$n$ and~$d$. (In fact, we
can take~$D(n,d)=(d+1)^n$.)

Suppose first that~$f\in\End_d^n(\Qbar)$.  Then Northcott's
theorem~\cite{northcott:periodicpoints} (or
see~\cite[Theorem~3.12]{MR2316407}) says that~$\Per(f)$ is a set of
bounded height, so Lemma~\ref{lemma:bdhtset} tells us that every
component of~$\Fix(f)$ has dimension~$0$, and since~$\Fix(f)$ has only
finitely many components, this completes the proof
that~$\dim\Fix(f)=0$.
\par
We next suppose that~$f\in\End_d^n(\FF)$ and that~$\dim\Fix(f)\ge1$,
so in particular the previous paragraph says that~$f$ is not defined
over~$\Qbar$.  Since~$\End_d^n$ is of finite type, the map~$f$ is
defined over some finitely-generated extension of~$\QQ$, which we may
take, without loss of generality, to be the function field~$K(X)$ of a
positive-dimensional variety~$X$ defined over a number field
$K$. Using the assumption that~$\dim\Fix(f)\ge1$, and replacing~$K(X)$
with a finite extension if necessary, we may also assume without loss
of generality that the number of~$K(X)$-rational fixed points of~$f$
is at least~$m$, for any given~$m$.
\par
We may specialize~$f$ at any point of~$X(\Kbar)$ to get a map defined
over~$\Kbar=\Qbar$. There exists a non-empty Zariski-open
set~$U\subseteq X$ so that for points~$x\in{U}(\Kbar)$, the
specialization~$f_x$ will still be an endomorphism of degree~$d$,
i.e.,~$f_x\in\End_d^n(\Kbar)$.  Further, if we choose any finite set
of distinct points of~$\PP^n$ over~$K(X)$, they will remain distinct
after specialization on a non-empty, Zariski-open set of points
$x\in{X}(\Kbar)$.
\par
Now suppose that there are at least~$D(n,d)+1$ distinct fixed points
of~$f$ defined over~$K(X)$. By the argument above, there is a
specialization~$f_x$ defined over~$\Kbar$ which is itself a morphism
of degree~$d$, and such that the~$D(n, d)+1$ fixed points above
specialize to~$D(n,d)+1$ distinct fixed points of~$f_x$. This
contradicts what we have already shown, and so we must
have~$\dim\Fix(f)=0$.
\end{proof}

\Step
We resume the proof of Theorem~\ref{theorem:PCPfinite}.
Let~$f$ be an element of the set
\begin{equation}
  \label{eqn:fPCPdnCC}
  \{f\in\PCP_d^n(\FF) : \text{$\Crit_f$ is
         irreducible and of general type} \}.
\end{equation} 
Applying Step~2, we find an integer~$m=m(f)\ge1$ so
that
\[
\Crit_f\subseteq\Fix(f^m).
\]
The map~$f^m$ is in~$\End_{d^m}^n(\FF)$, so applying
Lemma~\ref{lemma:fixdim0} to the map~$f^m$ tells us that
$\dim\Fix(f^m)=0$. Hence
\[
n-1 = \dim \Crit_f \le \dim \Fix(f^m) = 0,
\]
contradicting our assumption that~$n\ge2$.
\end{proof}

\section{Two multiplicity lemmas}
\label{section:twomultiplicitylemmas}
In this section we prove two multiplicity lemmas that will be used to
deal with PCF maps of tail length~$1$.

\begin{definition}
\label{definition:multiplicity}
We use~$\Mult$ to denote multiplicity in various contexts.  Thus
if~$s$ is a local parameter cutting out~$\Crit_f$ near~$p$ and~$t$ is
a local parameter cutting out~$\Branch_f$ near~$f(p)$, then
\[
f^{\#}(t)=(\text{unit~in~the~local~ring~at~$p$})\cdot{s}^k
\quad\text{with}\quad\Mult_{\Crit_f}(f) = k.
\]
And if~$Z$ is a zero-dimensional scheme and~$p\in Z$,
then~$\Mult_Z(p)$ is the scheme-theoretic multiplicity of~$Z$ at~$p$.
\end{definition}

\begin{lemma}
\label{multiplicityalongcriticallocus}
Let~$X$ and~$Y$ be projective varieties of dimension~$n$,
let~$f:X\to{Y}$ be a morphism, and let~$p\in \Crit_f$ be a point
satisfying\textup:
\begin{parts}
  \Part{\textbullet}
  $p$ is a smooth point of~$\Crit_f$.
  \Part{\textbullet}
  $p$ is a smooth point of~$X$.
  \Part{\textbullet}
  $f(p)$ is a smooth point of~$Y$.
  \Part{\textbullet}
  The restriction~$f|_{\Crit_f}$ is an immersion near~$p$.
\end{parts}
Then we have\textup:
\begin{parts}
  \Part{(a)}
  The point~$p$ is an isolated point of~$f^{-1}(f(p))$.
  \Part{(b)}
  The multiplicity of~$p$ in this set equals the multiplicity of~$f$
  along its critical locus,
\[
\Mult_{f^{-1}(f(p))}(p) = \Mult_{\Crit_f}(f).
\]
\end{parts}
\end{lemma}
\begin{proof}
We let
\[
k = \Mult_{\Crit_f}(f).
\]
We first note that since~$p$ is a smooth point of~$\Crit_f$ and
$f|_{\Crit_f}$ is an immersion near~$p$, it follows that~$f(p)$ is a
smooth point of~$\Branch_f$. We work in the completions of the local
rings at~$p$ and~$f(p)$, so we can pick local equations~$s$ cutting
out~$\Crit_f$ at~$p$ and~$t$ cutting out~$\Branch_f$ at~$f(p)$ such
that~$f^{\#}(t)= s^k$. We complete~$s$ and~$t$ respectively to local
coordinates~$(x_1,\ldots,x_{n-1},x_{n})=s$ for~$X$ at~$p$ and
$(y_1,\ldots,y_{n-1},y_{n})=t$ for~$Y$ at~$f(p)$ in such a way that
$(x_1,\ldots,x_{n-1})$ restrict to local coordinates for~$\Crit_f$ at
$p$, and~$(y_1,\ldots,y_{n-1})$ restrict to local parameters for
$\Branch_f$ at~$f(p)$, and further so that in these coordinates, the map induced by
$f_{\Crit_f}$ from the completion of the local ring of~$\Branch_f$ at
$f(p)$ to the completion of the local ring of~$\Crit_f$ at~$p$ is
\begin{align*}
f_{\Crit_f}^{\#}:\FF[\![y_1,\ldots,y_{n-1}]\!] &\to \FF[\![x_1,\ldots,x_{n-1}]\!], \\
y_i &\mapsto x_i, \qquad\text{$i=1,\ldots, n-1$.}
\end{align*} 
Then in these coordinates, the map induced by~$f$ from the completion
of the local ring of~$Y$ at~$f(p)$ to the completion of the local ring
of~$X$ at~$p$ is
\begin{align*}
f^{\#}:\FF[\![y_1,\ldots,y_{n-1}, y_n]\!] &\to \FF[\![x_1,\ldots,x_{n-1}, x_n]\!], \\
y_i &\mapsto
\begin{cases}
  x_i + f_i(x_n) &\text{for $i=1,\ldots, n-1$,}\\
  x_n^k &\text{for $i=n$,} \\
\end{cases}
\end{align*} 
where each~$f_i$ is a power series in~$x_n$ whose constant term is zero.

\begin{claim}
The set~$\{1,x_n,x_n^2,\ldots, x_n^{k-1}\}$ is an~$\FF$-basis for the vector sapce
\[
\frac{\FF[\![x_1,\dots, x_n]\!]}{\bigl(x_1 + f_1(x_n), \ldots, x_{n-1} + f_{n-1}(x_n), x_n^k\bigr)}.
\]
\end{claim}
\begin{proof}[Proof of Claim]
Both spanning and linear independence can easily be shown directly. 
\end{proof}

We conclude that 
\begin{multline*}
\frac{\FF[\![x_1,\dots, x_n]\!]}{\bigl(\text{pullback of maximal ideal of $f(p)$}\bigr)} \\
=\frac{\FF[\![x_1,\dots, x_n]\!]}{\bigl(x_1 + f_1(x_n), \ldots, x_{n-1} + f_{n-1}(x_n), x_n^k\bigr)}
\end{multline*}
has dimension~$k$ over~$\FF$, so~$p$ is an isolated point of multiplicity~$k$ in~$f^{-1}(f(p))$. 
\end{proof}

\begin{lemma}
\label{lem:multiplicityofcriticalpoint}
Let~$X$ and~$Y$ be projective varieties of dimension~$n$, let~$f:X\to Y$ be a morphism,
and let~$p\in{X}$ and~$q\in{Y}$ be
 smooth points such that~$p$ is an isolated point of
multiplicity~$k$ in~$f^{-1}(q)$. Let~$(x_1,\ldots{x}_n)$ be coordinates
at~$p$, so the completion of the local ring to~$X$ at~$p$ is
$\FF[\![x_1,\ldots, x_n]\!]$, and let~$(z_1,\ldots{z}_n)$ be coordinates at
$q$, so the completion of local ring to~$Y$ at~$q$ is
$\FF[\![z_1,\ldots, z_n]\!]$, and suppose that in these coordinates we
have~$z_i=f_i(x_1,\ldots, x_n)$. Denote the maximal ideals of the
completions of the local rings at~$p$ and~$q$ by~$\gm$ and~$\gn$ 
respectively.
\begin{parts}
  \Part{(1)}
  The following are equivalent\textup:
  \begin{parts}
    \Part{(A)} $k=1$.
    \Part{(B)} $f_1,\ldots, f_n$ generate $\gm$.
    \Part{(C)} $\{f_1,\ldots,f_n\} \bmod \gm^2$ is an $\FF$-basis for $\gm/\gm^2$.
    \Part{(D)} $p \not\in \Crit_f$.
  \end{parts}
  \Part{(2)}
  If~$k=2$, then the following are true\textup:
  \begin{parts}
    \Part{(a)}
    $p$ is a smooth point of $\Crit_f$.
    \Part{(b)}
    $f|_{\Crit_f}$ is an immersion near $p$.
    \Part{(c)}
    $f$ has multiplicity~$2$ along~$\Crit_f$ near~$p$.
  \end{parts}
\end{parts}
\end{lemma}
\begin{proof}
Recall that
\[
k=\dim_{\FF} \frac{\FF[\![x_1,\dots, x_n]\!]}{(f_1,\ldots, f_n)}.
\]
This implies the equivalence of (A) and (B). Nakayama's lemma implies
the equivalence of (B) and (C). By definition,~$p\in \Crit_f$ if and
only if the Jacobian of~$f$, i.e., the induced map on tangent spaces,
drops rank at~$p$. The Jacobian at~$p$ is dual to the induced map from
$\gn/\gn^2$ to~$\gm/\gm^2$. In turn, the map from~$\gn/\gn^2$ to
$\gm/\gm^2$ sends the basis~$\{z_1\ldots,z_n\}$ to
$\{f_1,\ldots,f_n\}\bmod\gm^2$. Thus the Jacobian at~$p$ is full rank
if and only if~$\{f_1,\ldots,f_n\}\bmod\gm^2$ is an~$\FF$-basis for
$\gm/\gm^2$, proving the equivalence of (C) and (D). This completes
the proof of Part~(1) of Lemma~\ref{lem:multiplicityofcriticalpoint}.

For Part~(2) we suppose that~$k=2$. By the preceding discussion, the
set~$\{f_1,\ldots,f_n\}\bmod\gm^2$ does not generate~$\gm/\gm^2$. Let
$g_1,\ldots,g_s$ be functions whose reductions modulo~$\gm^2$ form a
basis for
\[
\frac{\gm/\gm^2}{\Span\bigl(\{f_1,\ldots,f_n\} \bmod \gm^2\bigr)}.
\]
Note that we have that~$s\ge 1$. Also
\[
  \text{ $1,g_1,\ldots, g_s$ are linearly independent in 
  $\dfrac{\FF[\![x_1,\dots, x_n]\!]}{(f_1,\ldots,f_n)}$. }
\]
But
\[
\dim_\FF \frac{\FF[\![x_1,\dots, x_n]\!]}{(f_1,\ldots,f_n)} = 2,
\]
which implies that~$s=1$, and hence that~$1,g_1$ form a basis. We
conclude that~$\{f_1,\ldots,f_n\}\bmod\gm^2$ span an
$(n-1)$-dimensional subspace, so without loss of generality we may
assume that~$\{f_1,\ldots,f_{n-1}\}\bmod\gm^2$ are linearly
independent, and that~$\{f_1,\ldots,f_{n-1},g_1\}\bmod\gm^2$ is a
basis for~$\gm/\gm^2$. By Nakayama's lemma again,
\[
\{y_1,\ldots,y_n\} := \{f_1,\ldots,f_{n-1},g_1\}
\]
generate~$\gm$ and form an alternate system of coordinates at
$p$. With respect to these new coordinates,~$f_n$ is a power series
$f'_n$ in~$y_1,\ldots, y_n$. We expand~$f'_n$ with respect to the last
coordinate~$y_n$,
\[
f'_n(y_1,\ldots, y_n)=c_0 + c_1 y_n + c_2 y_n^2 + \cdots,
\]
where each~$c_i$ is a power series in~$y_1,\ldots, y_{n-1}$. Also 
\[
\frac{\partial f'_n}{\partial y_n}(y_1,\ldots, y_n)=c_1 + 2 c_2 y_n + 3 c_3 y_n^2 +\cdots\,.
\]
We know that~$1,y_n$ forms a basis for
\[
\frac{\FF[\![y_1,\ldots,y_n]\!]}{(y_1,\ldots, y_{n-1}, f'_n)}\cong\frac{\FF[\![y_n]\!]}{c_0(0,\ldots,0) + c_1(0,\ldots,0) y_n + c_2(0,\ldots,0) y_n^2 + \cdots},
\]
so we must have
\[
c_0(0,\ldots,0)=c_1(0,\ldots,0)=0
\quad\text{and}\quad
c_{2,0}:=c_2(0,\ldots,0)\ne 0.
\]
Let 
\[
c_1=c_{1,1} y_1 + \cdots +c_{1,n-1} y_{n-1} + (\text{higher order terms in $\gm^2$}),
\]
where each~$c_{1,i} \in \FF$. Then 
\[
\frac{\partial f'_n}{\partial y_n}= c_{1,1} y_1 + \cdots c_{1,n-1} y_{n-1} + 2 c_{2,0} y_n + (\text{something in $\gm^2$}).
\]

We want to re-write~$f$ in coordinates~$y_1,\ldots y_n$ at~$p$ and
$z_1, \ldots z_n$ at~$q$. We have the induced map on the completions
of local rings,
\begin{align*}
  f^{\#}:\FF[\![z_1,\ldots,z_n]\!] &\to \FF[\![y_1,\ldots,y_n]\!] \\
  z_i &\mapsto
  \begin{cases}
    y_i&\text{for $i=1,\ldots, n-1$,} \\
    f_n'&\text{for $i=n$.}
  \end{cases}
\end{align*} 

In these coordinates, the Jacobian matrix~$J_f$ is of the form 
\begin{align*}
J_f=
\begin{bmatrix}
1&0&0&\cdots&0& \frac{\partial f'_n}{\partial y_1}\\
0&1&0&\cdots&0& \frac{\partial f'_n}{\partial y_2}\\
0&0&1&\cdots&0& \frac{\partial f'_n}{\partial y_3}\\
\vdots &\vdots &\vdots &\ddots&\vdots &\vdots \\
0&0&0&\cdots&1& \frac{\partial f'_n}{\partial y_{n-1}}\\
0&0&0&\cdots&0&\frac{\partial f'_n}{\partial y_n} \\
\end{bmatrix}.
\end{align*}
The critical locus~$\Crit_f$ is locally cut out by the determinant of
$J_f$, which in these coordinates is
\[
\det(J_f)=\frac{d f'_n}{d y_n}= c_{1,1} y_1 + \cdots c_{1,n-1} y_{n-1} + 2 c_{2,0} y_n + (\text{something in $\gm^2$}).
\]
Since~$c_{2,0}$ is non-zero in~$\FF$, we see that~$\det(J_f)$ is non-zero in
$\gm/\gm^2$, which implies that~$\Crit_f$ is smooth at~$p$.

The tangent space to~$\Crit_f$ at~$p$ is cut out by the equation
\[
y_n= -\frac{c_{1,1}}{2 c_{2,0}} y_1 - \cdots - \frac{c_{1,n-1}}{2 c_{2,0}} y_{n-1},
\]
so~$y_1,\ldots, y_{n-1}$ restrict to give local coordinates (a basis)
for the cotangent space to~$\Crit_f$ at~$p$. The map
$f|_{\Crit_f}:\Crit_f\to Y$ induces the following map of completions of local
rings at~$p$ and~$q$:
\begin{align*}
f_{\Crit_f}^{\#}:\FF[\![z_1,\ldots,z_n]\!] &\to \frac{\FF[\![y_1,\ldots,y_{n}]\!]}{(\det(J_f))}\isom\FF[\![y_1,\ldots,y_{n-1}]\!] \\
z_i &\mapsto
\begin{cases}
  y_i&\text{for $i=1,\ldots,n-1$,} \\
  f_n' \bmod \det(J_f)&\text{for $i=n$.}\\
\end{cases}
\end{align*} 
In these coordinates, it is clear that the map on cotangent spaces is
surjective, so the map on tangent spaces is injective. Thus the map
$f|_{\Crit_f}:\Crit_f\to Y$ is an immersion near~$p$, as
desired. Finally, a direct application of Lemma
\ref{multiplicityalongcriticallocus} tells us that~$f$ has multiplicity~$2$
along~$\Crit_f$ near~$p$.
\end{proof}

\section{A map with a minimally branched point}
\label{section:hyperplaneconstruction}

In this section we construct a map~$f$ whose branch locus contains a
point that is minimally branched. We call this the ``hyperplance
construction'' because the coordinates of the map~$f$ that we
construct vanish along hyperplanes.

\begin{proposition}[Hyperplane Construction]
\label{proposition:hyperplane}
Let~$n\ge1$ and~$d\ge2$.
\begin{parts}
\Part{(a)}
There exists a morphism~$f:\PP^n\to\PP^n$ of degree~$d$ containing
a branch point $q\in \Branch_f$ with the property that
\begin{equation}
  \label{eqn:fq2ppluspi}
  f^{*}(q)= 2 p + \underbrace{p_1 + p_2 + \cdots + p_{d^n-2}}_{\text{\hidewidth distinct points different from $p$\hidewidth}}.
\end{equation}
\Part{(b)}
Let~$f$ be a map as in~\textup{(a)} with a point~$q$ satisfying~\eqref{eqn:fq2ppluspi}. Then the following are true\textup:
\begin{parts}
  \Part{(1)}
  The point~$q$ is a smooth point of~$\Branch_f$, and thus lies on exactly one irreducible component~$B$ of~$\Branch_f$.
  \Part{(2)}
  There exists a unique irreducible component~$C$ of~$\Crit_f$ mapping to~$B$.
  \Part{(3)}
  The map~$f|_{C}:C \to B$ is generically $1$-to-$1$.
  \Part{(4)}
  The map~$f$ has multiplicity~$2$ along~$C$.
\end{parts}
\end{parts}
\end{proposition}

\begin{proof}
(a)\enspace We take
\[
q=[0:0:\cdots:0:1]\in\PP^n,
\]
and we use~$\bfX=[X_1:\cdots:X_{n+1}]$ as homogeneous coordinates
on~$\PP^n$. We are going to create a map
\[
f(\bfX)=[f_1:\cdots:f_{n+1}]\quad\text{with}\quad f_i(\bfX)=\prod_{j=1}^d L_{i,j}(\bfX),
\]
where the~$L_{i,j}(\bfX)$ are linear forms that will be constructed inductively.
\par
We note that
\[
f(P)=q\quad\Longleftrightarrow\quad
\left(\begin{tabular}{@{}c@{}}
  for all $1\le i\le n$ there is some index \\ $1\le\s(i)\le d$ such that $L_{i,\s(i)}(P)=0$.\\
  \end{tabular}\right)
\]
In other words, the solutions to~$f(P)=q$ are parameterized
by the~$d^n$ functions
\[
\s:\{1,2,\ldots,n\} \longrightarrow \{1,2,\ldots,d\},
\]
where a given~$\s$ corresponds to the solution(s)~$P_\s$ to the system
of linear equations
\begin{equation}
\label{eqn:L1siLsns0}
L_{1,\s(1)}(P)=L_{2,\s(2)}(P)=\cdots=L_{n,\s(n)}(P)=0.
\end{equation}
To ease notation, we denote this set of index maps by
\[
\SetMap(n,d) :=
\bigl( \text{collection of maps $\s:\{1,\ldots,n\}\to\{1,\ldots,d\}$} \bigr).
\]
\par
We start our construction by setting
\[
\text{$L_{n+1,j}(\bfX)=X_{n+1}$ for all $1\le j\le d$,}
\]
i.e., we take
\[
f_{n+1}(\bfX):=X_{n+1}^d.
\]
This allows us to dehomogenize~$X_{n+1}=1$, and then by abuse of
notation, we write~$f=(f_1,\ldots,f_n)$ for the affine
map~$f:\AA^n\to\AA^n$ having affine coordinates~$(X_1,\ldots,X_n)$,
and~$q=(0,0,\ldots,0)$.
\par
We next assign the initial linear form in each~$f_i$ to be~$X_i$, i.e., 
\[
L_{1,1}=X_1,\; L_{2,1}=X_2,\ldots, L_{n,1}=X_n,
\;\text{and thus}\;
f_i=X_iL_{i,2}L_{i,3}\cdots L_{i,d}.
\]
\par
The next step is to select the second linear form in~$f_1$, which we do by
setting
\[
L_{1,2}=X_1-X_2. \quad\text{Thus}\quad
f_1 = X_1(X_1-X_2)L_{1,3}\cdots L_{1,d}.
\]
\par
This allows us to determine the solution~$P_\s$
to~\eqref{eqn:L1siLsns0} for the following two particular index
maps~$\s_1$ and~$\s_2$ in~$\SetMap(n,d)$:
\begin{align*}
  \s_1\in\SetMap(n,d)&\text{ is defined by $\s_1(i)=1$ for all $1\le i\le n$.} \\
  \s_2\in\SetMap(n,d)&\text{ is defined by $\s_2(i)=\begin{cases} 2&\text{if $i=1$,}\\ 1&\text{for $2\le i\le n$.}\\ \end{cases}$}
\end{align*}
For these index maps we have
\begin{align*}
  P_{\s_1} &= \{X_1=X_2=X_3=\cdots=X_n=0\} = q, \\
  P_{\s_2} &= \{X_1-X_2=X_2=X_3=\cdots=X_n=0\} = q.
\end{align*}  
\par

Now suppose that for a given $k_1,\ldots,k_n\in\{1,\ldots,d\}$, we
have constructed linear forms
\begin{align*}
  L_{1,1},&\ldots,L_{1,k_1},\\
  L_{2,1},&\ldots,L_{2,k_2},\\
  &\vdots \\
  L_{n,1},&\ldots,L_{n,k_n},  
\end{align*}
such that for every
\begin{equation}
\label{eqn::sinSetMapndleki}
\s\in\SetMap(n,d)\quad\text{satisfying $\s(i)\le k_i$ for all $1\le i\le n$,}
\end{equation}
the following hold:
\begin{parts}
  \Part{\textbullet}
  There is a solution $P_\s$ to~\eqref{eqn:L1siLsns0}.
  \Part{\textbullet}
  The solutions $P_\s$ corresponding to the~$\s$
  satisfying~\eqref{eqn::sinSetMapndleki} are distinct except for the
  duplicate value $P_{\s_1}=P_{\s_2}=q$ noted earlier.
\end{parts}
\par
Suppose that
\[
\text{$k_t<d$ for some $1\le{t}\le{n}$.}
\]
Then we choose a linear form $L_{t,k_t+1}$ such that
\[
\text{$L_{t,k_t+1}(P_\s)\ne0$ for all $\s$ satisfying \eqref{eqn::sinSetMapndleki},}
\]
i.e., we want~$L_{t,k_t+1}$ to not vanish at all of the previously
selected points.  We can find such a linear form by choosing a point
in the dual space~$\check{\PP}^n$ that is not on any of the
hyperplanes defined by the previously selected~$P_\s$. (This is where
we use the assumption that our field~$\FF$ is infinite, since it
ensures that~$(\check{\PP}^n)(\FF)$ is not covered by finitely many
hyperplanes.)
\par
Note that it also follows that for all~$\s$
satisfying~\eqref{eqn::sinSetMapndleki}, the
hyperplane~$L_{t,k_t+1}=0$ does not contain the line
\[
\bigcap_{i\ne t} L_{i,\s(i)},
\]
since if it did, then the form~$L_{t,k_t+1}$ would vanish at all
points on this line, including~$P_\s$. Hence for every~$\s$ satisfying
\[
\text{$\s(i)\le k_i$ for $i\ne t$  \quad and\quad $\s(t)=k_t+1$,}
\]
the hyperplanes
\[
L_{1,\s(1)},\, L_{2,\s(2)},\ldots, L_{n,\s(n)}
\]
intersect properly at a point~$P_\s$ that cannot equal any of the
previously constructed points.
\par
Continuing this process, we end up with linear forms
\[
\text{$L_{i,j}$ for all~$1\le{i}\le{n}$ and all~$1\le{j}\le{d}$}
\]
such that for $\s,\t\in\SetMap(n,d)$, we have
\[
P_{\s}=P_{\t} \;\Longleftrightarrow\; \s=\t~\text{or}~\{\s,\t\}=\{\s_1,\s_2\},
\]
where~$\s_1$ and~$\s_2$ are the maps defined earlier.
It follows that the map
\[
f(\bfX) := \biggl[
  \prod_{j=1}^d L_{1,j}(\bfX) : \cdots : \prod_{j=1}^d L_{n,j}(\bfX) : X_{n+1}^d
  \biggr],
\]
satisfies
\[
f^{*}(q)= 2 q + p_1 + p_2 + \cdots + p_{d^n-2},
\]
where the points~$q,p_1,\ldots,p_{d^n-2}$ are distinct.  This
completes the proof of Proposition~\ref{proposition:hyperplane}(a).
\par\noindent(b)\enspace  
Lemma~\ref{lem:multiplicityofcriticalpoint} tells us that:
\begin{enumerate}
\item[\textbullet]
  $p$ is the only point on $\Crit_f$ that maps to $q$.
\item[\textbullet]
  $p$ is a smooth point of $\Crit_f$.
\item[\textbullet]
  The map $f|_{\Crit_f}:\Crit_f \to \PP^n$ is an immersion near~$p$.
\end{enumerate}
This implies that~$\Branch_f$ is smooth at~$q$, so~$q$ lies on a
unique irreducible component of~$\Branch_f$, as desired. We know
already that~$q\in B$ has exactly one pre-image point in~$\Crit_f$,
and that that pre-image point~$p$ is a smooth point of~$\Crit_f$,
which implies that the unique irreducible component~$C$ of~$\Crit_f$
containing~$p$ is the only irreducible component of~$\Crit_f$ mapping
to~$B$. Since~$f|_{C}:C \to \PP^n$ is an immersion near~$p$, it is
generically $1$-to-$1$. Finally,
Lemma~\ref{lem:multiplicityofcriticalpoint} also tells us that~$f$ has
order~$2$ along~$C$, which completes the proof of
Proposition~\ref{proposition:hyperplane}(b).
\end{proof}

\begin{remark}
With minor modifications, the proof of
Proposition~\ref{proposition:hyperplane}(a) can be modified to
construct a map satisfying~$f^{*}(q)=ep+p_1+p_2+\cdots+p_{d^n-e}$ for
any~$e\ge2$. To do this, in the proof we simply start by
choosing~$L_{1,2},\ldots,L_{1,e}$ to be linear forms defining
hyperplanes in general position.
\end{remark}

\section{PCF maps with fixed tail length}
\label{section:fixedtaillength}
In this section, we prove a number of results about PCF maps with
fixed tail length~$\ell$. An immediate consequence will be a proof
that PCF maps with tail length~$1$ are sparse, and the methods that we
develop will then be used in Section~\ref{section:taillength2} to show that PCF maps
with tail length at most~$2$ are sparse.

We recall that in Section~\ref{section:pbpmapssparse} we proved that a
map~$f$ whose critical locus is irreducible and of general type cannot
be PCP. The key to the proof is the fact that these assumptions imply
that some iterate~$f^m$ is an endomorphism of~$\Crit_f$, and hence is
an automorphism of finite order, since varieties of general type have
finite automorphism groups.

More generally, suppose that~$f$ is PCF of type~$(k,\ell)$. Then
$f^{k}$ restricts to an endomorphism of~$f^{\ell}(\Crit_f)$, but if
$f^{\ell}(\Crit_f)$ is not general type, then it may admit
endomorphisms that are not of finite order. On the other hand, by
Theorem~\ref{theorem:critgentype}, we know that for most maps~$f$,
the critical locus~$\Crit_f$ is of general type.  Our next proposition
lays out a roadmap for proving that PCF maps with fixed tail
length~$\ell$ are sparse.  It says, roughly, that such maps are sparse
provided that we can find even a single map~$f$ with the property
that~$f^{\ell}(\Crit_f)$ is of general type. Using this proposition,
we will easily be able to handle the case~$\ell=1$, and with
significantly more work as described in
Section~\ref{section:taillength2}, the case~$\ell=2$.

\begin{proposition}
\label{prop:taillengthell}
Let~$n\ge2$ and~$d\ge 3$ and~$\ell\ge1$. Suppose that there exists at
least one endomoprhism~$f_0\in\End_d^n$ such
that~$f_0^{\ell}(\Crit_{f_0})$ has an irreducible component~$B$ with
the following properties\textup:
\begin{parts}
  \Part{(1)}
  There is exactly one irreducible component~$C$ of~$\Crit_{f_0}$ satisfying
  \[
  f_0^{\ell}(C)=B.
  \]
   \Part{(2)}
  None of the images~$f_0(C),\ldots,f_0^{\ell-1}(C)$ is contained in~$\Crit_{f_0}$.
  \Part{(3)}
  The map~$f_0^{\ell}|_{C}:C \to B$ is generically $1$-to-$1$.
  \Part{(4)}
  The map~$f_0^{\ell}$ has multiplicity~$2$ along~$C$.
\end{parts}
Then the following are true\textup:
\begin{parts}
\Part{(a)}
There is a non-empty Zariski open subset~$U_{d,\ell}^n\subset \End_d^n$ such that for all~$f\in U_{d,\ell}^n$\textup:
\begin{parts}
  \Part{\textbullet}
 $\Crit_f$ is irreducible and of general type.
  \Part{\textbullet}
  The map~$f^{\ell}|_{\Crit_f}:\Crit_f \to f^{\ell}(\Crit_f)$ is generically $1$-to-$1$.
   \Part{\textbullet}
  The map~$f$ is \textbf{not} PCF with tail-length~$\ell$.
\end{parts}
\Part{(b)}
The set of PCF maps with exact tail length~$\ell$ is not Zariski dense in~$\End_d^n$.
\end{parts}
\end{proposition}

We start with some preliminary results.

\begin{lemma}
\label{lemma:genericdegree}  
Let~$n\ge3$ and~$d\ge 3$ and~$\ell\ge1$. There exists a positive
integer~$r_{d,\ell}^n$ and a non-empty Zariski-open set
$U_{d,\ell}^n\subset \End_d^n$ such that every~$f\in U_{d,\ell}^n$ has
the following properties\textup:
\begin{parts}
  \Part{(1)}
  The critical locus~$\Crit_f$ is irreducible and of general type.
  \Part{(2)}
  The map~$f^{\ell}|_{\Crit_f}:\Crit_f\to f^{\ell}(\Crit_f)$ is generically~$r_{d,\ell}^n$-to-$1$.
\end{parts}
\end{lemma}
\begin{proof}
We first observe that there is a non-empty Zariski open set
$(U_{d,\ell}^n)_1\subset \End_d^n$ such that for all~$f\in(U_{d,\ell}^n)_1$:
\begin{parts}
  \Part{(1)}
  The critical locus~$\Crit_f$ is irreducible and of general type.
  The fact that this is a non-empty open condition follows from
  Theorem~\ref{theorem:critgentype}.
  \Part{(2$'$)}
  The maps~$f,f^2,\ldots, f^{\ell}$ have no non-trivial automorphisms.\footnote{In general, the
    automorphism group of a dynamical system~$f:\PP^n\to\PP^n$ is
    $\Aut(f):=\{\a\in\PGL_{n+1}:\a\circ{f}=f\circ\a\}$. It is proven in~\cite{MR2741188}
    that if~$f$ is a morphism and~$d\ge2$, then~$\Aut(f)$ is finite, and that
    the set of~$f\in\End_d^n$ with~$\Aut(f)\ne1$ is a Zariski closed set.}
  The fact that this is a non-empty open condition follows from~\cite{MR2741188}.
\end{parts}
Then over~$(U_{d,\ell}^n)_1$ there is  a universal family
\[
\mathcal{F}: \PP^n\times (U_{d,\ell}^n)_1\to\PP^n\times (U_{d,\ell}^n)_1,
\] 
with universal critical locus~$\hat\Crit\to(U_{d,\ell}^n)_1$. We
denote by~$\hat{\mathcal{B}}_{\ell}$ the underlying reduced variety of the
image~$\mathcal{F}^{\ell}(\hat\Crit)$ of the universal critical locus
under the~$\ell$th iterate of~$\mathcal{F}$.

The restriction~$\mathcal{F}^{\ell}|_{\hat\Crit}:\hat\Crit\to\hat{\mathcal{B}}_{\ell}$ is
generically finite, so has some generic degree~$r_{d,\ell}^n$. There is an
open set~$(U_{d,\ell}^n)_2\subset (U_{d,\ell}^n)_1$ over which~$\pi_{\hat{\mathcal{B}}_{\ell}}$
is flat, as well as an open set~$\hat{\mathcal{B}}_{\ell}^{\circ}\subset
\pi_{\hat{\mathcal{B}}_{\ell}}^{-1}\bigl((U_{d,\ell}^n\bigr)_2)$ over which
$\mathcal{F}^{\ell}|_{\hat\Crit}$ is \'etale of degree exactly~$r_{d,\ell}^n$. Since
a flat map of finite type of Noetherian schemes is open, the set
\[
U_{d,\ell}^n:=\pi_{\hat{\mathcal{B}}_{\ell}}(\hat{\mathcal{B}}_{\ell}^{\circ})\subset (U_{d,\ell}^n)_2
\]
is open, and over~$U_{d,\ell}^n$, the map
\[
\mathcal{F}^{\ell}|_{\hat\Crit}:\hat\Crit\to\hat{\mathcal{B}}_{\ell}
\]
has generic degree~$r_{d,\ell}^n$ by construction.
\end{proof}

\begin{lemma}
\label{lem:gendegone}
Suppose that there is an endomorphism~$f_0\in\End_d^n$ satisfying the
hypothesis~$(1)$--$(4)$ of Proposition~$\ref{prop:taillengthell}$.
Then the degree~$r_{d,\ell}^n$ described in
Lemma~$\ref{lemma:genericdegree}$ satisfies~$r_{d,\ell}^n=1$.
\end{lemma}
\begin{proof}
We are given a map~$f_0$ that satisfies the four hypotheses of
Proposition~\ref{prop:taillengthell}.
Since~$f_0$ is in the closure of~$U_{d,\ell}^n$, we can find a map
\[
F:\Spec\bigl(\FF\dspt\bigr)\to \End_d^n
\]
such that the generic point~$\Spec\bigl(\FF\drpt\bigr)$ maps to~$U_{d,\ell}^n$ and the
special point at~$t=0$ maps to~$f_0$. Taking a ramified base change
if necessary, we obtain from~$F$ a family of degree~$d$ morphisms over
$\Spec\bigl(\FF\dspt\bigr)$,
\[
\Fcal:\Spec\bigl(\FF\dspt\bigr)\times \PP^n \to \Spec\bigl(\FF\dspt\bigr)\times \PP^n.
\]

Denote by~$\hat\Crit$ the underlying reduced scheme of the critical
locus of~$\Fcal$. It has pure codimension one. Denote by
$\hat\Branch_\ell$ the underlying reduced variety of the image
$\Fcal^{\ell}(\hat\Crit)$ of the universal critical locus under
the~$\ell$th iterate of~$\Fcal$. Denote by~$\Fcal_{\eta}$
the restriction of~$\Fcal$ to the generic fiber
$\Spec\bigl(\FF\drpt\bigr)\times \PP^n$ and by~$\Fcal_{0}$ the
restriction of~$\Fcal$ to the special fiber~$\PP^n_{\FF}$. By
construction, we have:
\begin{parts}
  \Part{\textbullet}
  $\Fcal_{0}=f_0$
  \Part{\textbullet}
  $\hat\Crit_{\eta}$ is irreducible general type.
  \Part{\textbullet}
  $\Fcal^{\ell}_{\eta}|_{\mathcal{\Crit_{\eta}}}$ has degree~$r_{d,\ell}^n$.
\end{parts}
Further, since~$\deg(\Fcal_0)=\deg(\Fcal)=d$, the map
$\Fcal$ is not ramified along the special fiber. We conclude
that~$\hat\Crit$ is the Zariski closure of~$\hat\Crit_\eta$ and that
$\hat\Branch_{\ell}$ is the Zariski closure of~$(\hat\Branch_\ell)_{\eta}$.

Let~$p \in C$ be a smooth point such that:
\begin{itemize}
\item The points $f_0(p),f_0^2(p),\ldots,f_0^{\ell}(p)=q$ are not in the critical locus~$\Crit_{f_0}$.
\item The point $p$ is not in the critical loci of any of the restrictions
  \[
  (f_0)|_C,(f_0^2)|_C,\ldots (f_0^{\ell})|_{C}.
  \]
\end{itemize}
Then $p$ and~$q=f_0^{\ell}(p)$ satisfy the conditions in
Proposition~\ref{proposition:hyperplane}(a) with respect to
$f_0^\ell$, that is, the divisor~$(f_0^\ell)^{*}(q)$ is the sum of
$2p$ and~$(d^\ell)^n-2$ points having multiplicity~$1$.

On the one hand,~$(\Fcal_0^\ell)^{-1}(q)$ is a subscheme
of~$(\Fcal^\ell)^{-1}(q)$, while on the other hand, both schemes have
degree $(d^\ell)^n$
over~$\FF$. Therefore~$(\Fcal_0^\ell)^{-1}(q)=(\Fcal^\ell)^{-1}(q)$. This
means that~$p$ has multiplicity exactly~$2$ in
$(\Fcal^\ell)^{-1}(q)$. Since the proof of Lemma
\ref{lem:multiplicityofcriticalpoint} was local, we conclude that
$\hat\Crit$ is smooth at~$p$, and that~$\hat\Branch_\ell$ is smooth
at~$q$. We also have that~$(\hat\Branch_\ell)_0$ is smooth at~$q$.

\begin{claim}
  The following are true:
  \begin{parts}
    \Part{\textbullet}
    The map~$\pi_{\hat\Crit}:\hat\Crit\to \Spec\bigl(\FF\dspt\bigr)$ is smooth at~$p$.
    \Part{\textbullet}
    The map~$\pi_{\hat\Branch_\ell}:\hat\Branch_\ell\to\Spec\bigl(\FF\dspt\bigr)$ is smooth at~$q$.
  \end{parts}
\end{claim}

\begin{proof}[Proof of Claim]
We follow the proof of
Lemma~\ref{lem:multiplicityofcriticalpoint}. Let~$(t,x_1,\ldots x_n)$
be coordinates at~$p$, so the completion of the local ring to
$\Spec\bigl(\FF\dspt\bigr)\times\PP^n$ at~$p$ is~$\FF[\![t, x_1,\ldots, x_n]\!]$,
and let~$(z_1,\ldots z_n)$ be coordinates at~$q$, so the completion of
the local ring to~$\Spec\bigl(\FF\dspt\bigr)\times\PP^n$ at~$q$ is
$\FF[\![t,z_1,\ldots,{z}_n]\!]$. Using these coordinates, we suppose that
$\Fcal^\ell$ is given by
\[
z_i=f_i(t, x_1,\ldots, x_n).
\]
Without loss of generality, we may assume that
\[
z_i=f_i(t, x_1,\ldots, x_n)=x_i\quad\text{for~$i=1, \ldots, n-1$.}
\]
As in the proof of Lemma~\ref{lem:multiplicityofcriticalpoint}, we
conclude that~$(t, x_1,\ldots, x_{n-1})$ restrict to local coordinates
on~$\hat\Crit$, and that~$(t, z_1,\ldots, z_{n-1})$ restrict to local
coordinates on~$\hat\Branch_\ell$. In these coordinates, the maps
$\pi_{\hat\Crit}$ and~$\pi_{\hat\Branch_\ell}$ are obtained, respectively,
by forgetting all of the~$x_i$ and~$z_i$ coordinates, and thus they
are smooth maps. This completes the proof of the claim.
\end{proof}

We resume the proof of Lemma~\ref{lem:gendegone}. The claim implies
that there exists a section
\[
P:\Spec\bigl(\FF\dspt\bigr)\to\hat\Crit
\]
with~$P(0)=p$. Then~$Q:=\Fcal^\ell\circ P$ is a section of
$\hat\Branch$. Since~$P_\eta\in\hat\Crit_\eta$, we see that~$P_\eta$
appears in~$(\Fcal_{\eta}^\ell)^{-1}(Q_\eta)$ with multiplicity at
least~$2$. On the other hand, by construction we know that
$(\Fcal^\ell)^{-1}(Q)|_{t=0}$ has~$d^n-1$ distinct~$\FF$-points, and
that~$(d^\ell)^n-2$ of them appear with multiplicity exactly~$1$. Hence
$(\Fcal_{\eta}^\ell)^{-1}(Q_\eta)$ must have at least~$(d^\ell)^n-2$ distinct
$\FF\drpt$-points appearing with multiplicity exactly~$1$. Therefore
$(\Fcal_{\eta}^\ell)^{-1}(Q_\eta)$ must have exactly~$(d^\ell)^n-1$
distinct~$\FF\drpt$-points, with exactly one of them,~$P_\eta$, appearing with
multiplicity~$2$. 
Proposition~\ref{proposition:hyperplane}(b)
implies that
$(\Fcal_\eta^\ell)|_{\hat\Crit_{\eta}}$ has degree~$1$, so~$r_{d,\ell}^n=1$,
as desired.
\end{proof}

We can now finish the proof of Proposition~\ref{prop:taillengthell}.

\begin{proof}[Proof of  Proposition~$\ref{prop:taillengthell}$]
Suppose that, for some fixed~$\ell$, the hypotheses of Proposition
\ref{prop:taillengthell} are satisfied. Then, by Lemmas
\ref{lemma:genericdegree} and~\ref{lem:gendegone}, there
is a non-empty Zariski open subset~$U_{d,\ell}^n\subset \End_d^n$ such
that for all~$f\in U_{d,\ell}^n$:
\begin{parts}
  \Part{\textbullet}
 $\Crit_f$ is irreducible and of general type.
  \Part{\textbullet}
  The map~$f^{\ell}|_{\Crit_f}:\Crit_f \to f^{\ell}(\Crit_f)$ is generically $1$-to-$1$.
\end{parts}
It remains to show that if~$f\in U_{d,\ell}^n$, then~$f$ is not PCF of
tail-length~$\ell$. Suppose we have some~$f\in U_{d,\ell}^n$. Then
$\Crit_f$ is irreducible and of general type, and since
$f^{\ell}|_{\Crit_f}:\Crit_f \to f^{\ell}(\Crit_f)$ is generically
$1$-to-$1$, we know that~$f^{\ell}(\Crit_f)$ is birational to
$\Crit_f$, and hence~$f^{\ell}(\Crit_f)$ is irreducible and of general
type. Assume for contradiction that~$f$ is PCF of tail-length~$\ell$
and some period~$k>0$. Then~$f^k$ defines an endomorphism of
$f^{\ell}(\Crit_f)$. As in Step 2 of Theorem~\ref{theorem:PCPfinite},
we conclude that~$f^k|_{f^{\ell}(\Crit_f)}$ is a finite-order
automorphism. Thus there exists some~$r>0$ such that
$f^{kr}|_{f^{\ell}(\Crit_f)}$ is the identity, i.e., such that
$\Crit_f\subseteq\Fix(f^{kr})$. But~$\Crit_f$ is a hypersurface, so it
has dimension~$n-1\ge1$, while Lemma~\ref{lemma:fixdim0} tells us that
$\Fix(f^{kr})$ has dimension~$0$.  The contradiction completes the
proof of Proposition~\ref{prop:taillengthell}.
\end{proof}

It is now a simple matter to prove that PCF maps with tail length
$\ell=1$ are sparse.

\begin{theorem}
\label{thm:Preperiodone}
Let~$n\ge3$ and~$d\ge3$. Then
\[
\{f\in \End_d^n :\text{$f^{k}(\Crit_f)\subseteq f(\Crit_f)$ for some $k\ge2$} \}
\]
is contained in a proper closed subvariety of~$\End_d^n$.
\end{theorem}
\begin{proof}
The map constructed in Proposition~\ref{proposition:hyperplane}
satisfies the hypotheses of Proposition~\ref{prop:taillengthell} for
$\ell=1$. We conclude that there is a non-empty Zariski open subset
$U_{d,1}^n\subset\End_d^n$ such that for all~$f\in U_{d,1}^n$, the
map~$f$ is not PCF of tail-length~$1$.
\end{proof}

\section{PCF maps with tail-length~$2$ are sparse}
\label{section:taillength2}

The main result of this section is as stated in the title. As in the
previous section, we begin with a number of preliminary results.

\begin{lemma}
\label{lemma:rto1tosto1}
Let~$n\ge2$, and let~$f:\PP^n\to\PP^n$ be a morphism of degree
$d\ge2$. Suppose that~$H\subset\PP^n$ is a hypersurface satisfying\textup:
\begin{parts}
  \Part{\textbullet}
 $H$ is not contained in~$\Crit_f$.
  \Part{\textbullet}
  $f(H)$ is not contained in~$\Branch_f$.
  \Part{\textbullet}
  $f|_H$ is generically~$r$-to-$1$ for some~$r\ge2$.
\end{parts}
Then there exists an~$s<r$ and an automorphism~$\alpha\in\Aut(\PP^n)$ such that
$f|_{\alpha(H)}$ is generically~$s$-to-$1$. 
\end{lemma}

\begin{remark}
Applying Lemma~\ref{lemma:rto1tosto1} repeatedly, we see that there
exists an~$\alpha\in\Aut(\PP^n)$ such that~$f|_{\alpha(H)}$ is
generically~$1$-to-$1$.
\end{remark}

\begin{proof}[Proof of Lemma~$\ref{lemma:rto1tosto1}$]
Let~$e=\deg(H)$. Then~$f_*([H])=r[f(H)]$ is~$d^{n-1}e$ times the class
of a hyperplane, so
\[
\text{$f(H)$ is a hypersurface of degree~$D:=\dfrac{d^{n-1}e}{r}$,}
\]
where for notational convenience we let~$D$ denote the frequently
appearing quantiy~$D=D(d,n,e,r):=d^{n-1}e/r$.

We pick a line~$L$ such that the intersection~$L\cap f(H)$
has the following properties:
\begin{parts}
  \Part{\textbullet}
  $L$ and~$f(H)$ intersect transversally.
  \Part{\textbullet}
  The intersection consists of exactly~$\ddd$ smooth points of~$f(H)$, say
  \[
  L\cap f(H) = \{q_1,\ldots,q_\ddd\}.
  \]
  \Part{\textbullet}
  $L\cap f(H)\cap\Branch_f=\emptyset$, i.e.,~$q_i\notin\Branch_f$ for all~$1\le i\le\ddd$.
  \Part{\textbullet}
  $L\cap f(H)\cap f(\text{singular locus of~$H$})=\emptyset$.
\end{parts}
It is possible to find such a line~$L$ because
the ``bad locus'' that we must avoid has codimension at least 2 in~$\PP^n$.

By construction,~$L$ is not contained in~$\Branch_f$, so~$f^{-1}(L)$
is a curve~$C$ of degree~$d^{n-1}$. Also, the intersection~$C\cap H$ is transversal,
consisting of exactly~$d^{n-1}e=rD$ smooth points of~$H$, which we label as
\[
  C\cap H = \{ p_{i,j} : 1\le i\le\ddd,\; 1\le j\le r \}
\]
so that:
\begin{align*} 
p_{1,1}, \ldots, p_{1,r} \quad &\text{map to}\quad q_1\\[-2\jot]
&\vdots\\[-1\jot]
p_{i,1}, \ldots, p_{i,r} \quad &\text{map to}\quad q_i\\[-2\jot]
&\vdots\\[-1\jot]
p_{\ddd,1}, \ldots, p_{\ddd,r} \quad &\text{map to}\quad q_\ddd.
\end{align*}

Without loss of generality, we may assume that
\[
p_{1,1}=[1:0:\cdots:0] \quad\text{and}\quad p_{1,2}=[0:1:0\cdots:0].
\]
For all~$i$ and~$j$, the point~$p_{i,j}$ is not in the branch locus
of~$f$, so~$f$ induces isomorphisms of completions of local rings
of~$\PP^n$. Writing~$\Rcal_p$ for the completion of the local ring
at~$p$, we have
\begin{align*}
  f_{i,j} :  \Rcal_{p_{i,j}} &\longrightarrow \Rcal_{q_{i}}, \\
  f_{i,j_1,j_2}:= f_{i,j_2}^{-1}\circ f_{i,j_1} :  \Rcal_{p_{i,j_1}} &\longrightarrow \Rcal_{p_{i,j_2}}.
\end{align*}

We pick a local parametrization of~$C$ near~$p_{1,1}$,
i.e., we fix a map
\[
P_{1,1}:\Spec\bigl(\FF\dspt\bigr)\to C \quad\text{with}\quad P_{1,1}(0)=p_{1,1}
\]
that induces an isomorphism between~$\FF\dspt$ and the completion of
the local ring of~$C$ at~$p_{1,1}$. We then obtain a local parametrization
of~$C$ near~$p_{1,2}$ as follows: First we pre-compose~$P_{1,1}$ with
a specified involution of~$\Spec\bigl(\FF\dspt\bigr)$, then we apply
$f_{1,1,2}$. Specifically, we set
\begin{equation}
  \label{eqn:P12tf112P11negt}
  P_{1,2}(t)=f_{1,1,2}(P_{1,1}(-t)),
\end{equation}
and then
\begin{align}
  \label{eqn:fP12t}
  (f&\circ P_{1,2})(t) \notag\\
  &= (f\circ f_{1,1,2} \circ P_{1,1})(-t) \quad\text{from \eqref{eqn:P12tf112P11negt},}\notag\\
  &= \bigl(f\circ (f_{1,2})^{-1}\circ f_{1,1} \circ P_{1,1}\bigr)(-t)
  \quad\text{since $f_{1,1,2}:=(f_{1,2})^{-1}\circ f_{1,1}$,} \notag\\
  &= (f_{1,1} \circ P_{1,1})(-t)
  \quad\text{since $f\circ (f_{1,2})^{-1}=\Id$ on $U_{1,2}$,} \notag\\
  &= (f \circ P_{1,1})(-t)
  \quad\text{since $f_{1,1}=f$ on $U_{1,1}$.}
\end{align}
We note that ${\frac{d}{dt}}f(P_{1,2}(t))\bigr|_{t=0}\ne0$, so taking
derivatives of~\eqref{eqn:fP12t} and evaluating at~$t=0$ yields
\[
  0
  \ne {\frac{d}{dt}}(f\circ P_{1,2})(t)\Bigr|_{t=0}
  = {\frac{d}{dt}} (f\circ P_{1,1})(-t))\Bigr|_{t=0}
  = -{\frac{d}{dt}} (f\circ P_{1,1})(t))\Bigr|_{t=0}  .
\]

The condition on~$t$ that the points
\[
P_{1,1}(t),\, P_{1,2}(t),\, [0:0:1:0\cdots:0],\,\ldots,\,[0:\cdots:0:1],\, [1:1\cdots:1]
\]
are in general position is an open condition
that is satisfied at~$t=0$, and thus it is satisfied over
$\Spec\bigl(\FF\dspt\bigr)$.

There is thus a unique element~$\alpha_t\in\PGL_{n+1}\bigl(\FF\dspt\bigr)$
satisfying
\begin{align*}
\alpha_t([1:0:0:0:\cdots:0:0])&=P_{1,1} \\
\alpha_t([0:1:0:0:\cdots:0:0])&=P_{1,2} \\
\alpha_t([0:0:1:0:\cdots:0:0])&=[0:0:1:0:\cdots:0:0] \\
\vdots & \qquad\vdots \\
\alpha_t([0:0:0:0:\cdots:0:1])&=[0:0:0:0:\cdots:0:1] \\
\alpha_t([1:1:1:1:\cdots:1:1])&=[1:1:1:1:\cdots:1:1]. 
\end{align*}
We note that~$\a$ has the following properties:
\begin{align}
  &\alpha_0 = \Id\in\PGL_{n+1}(\FF). \\
  &\alpha_t(p_{1,1})\in C\bigl(\FF\dspt\bigr)\quad\text{and}\quad \alpha_t(p_{1,2})\in C\bigl(\FF\dspt\bigr).\\
  &0\ne{\frac{d}{dt}}f(\alpha_t(p_{1,1}))\Bigr|_{t=0}= -{\frac{d}{dt}}(f(\alpha_t(p_{1,2}))\Bigr|_{t=0}.
  \label{eqn:ddteqnegddt}
\end{align}
Condition~\eqref{eqn:ddteqnegddt} implies that for~$t\ne 0$, i.e.,
over the generic point $\Spec\bigl(\FF\drpt\bigr)$, we have
\[
f(\alpha_t(p_{1,1}))\ne f(\alpha_t(p_{1,2})).
\]
We conclude that $f(\alpha_t(p_{1,1}))$ and~$f(\alpha_t(p_{1,2}))$
restrict to distinct points of~$L\bigl(\FF\drpt\bigr)$.
 
We can parametrize the intersection points of
$\bigl(\alpha_t(H)\cap{C}\bigr)\bigl(\FF\dspt\bigr)$, i.e., we can find maps
\[
P_{i,j}:\Spec\bigl(\FF\dspt\bigr)\to \alpha_t(H)\cap C
\]
such that~$P_{i,j}(0)=p_{i,j}$ for all~$i,j$. We have that
\begin{align*}
  P_{1,1}(t)&=\alpha_t(p_{1,1}),\\
  P_{1,2}(t)&=\alpha_t(p_{1,2}), \\
  f\circ P_{i,j} &\in (f(\alpha_t(H))\cap L)(\Spec\bigl(\FF\dspt\bigr)).
\end{align*}
The conditions on~$t$ that
\[
  f\circ P_{i,1} \ne f\circ P_{1,1}
  \quad\text{and}\quad
  f\circ P_{i,1} \ne f\circ P_{1,2}
  \quad\text{for all $2\le i\le \ddd$}
 \]
are open conditions satisfied at~$t=0$, and thus are satisfied over
$\FF\dspt$. On the other hand, for~$t\ne 0$, i.e., over
$\Spec\bigl(\FF\drpt\bigr)$, we have
\[
f\circ P_{1,1}=f(\alpha_t(p_{1,1}))\ne
f(\alpha_t(p_{1,2})) =f\circ P_{1,2}.
\]
Thus~$(f(\alpha_t(H))\cap L)(\Spec\bigl(\FF\drpt\bigr))$ contains at least
$\ddd+1$ distinct points, specifically
\[
f\circ P_{1,1}, f\circ P_{2,1},\ldots,f\circ P_{\ddd,1}, f\circ P_{1,2}
\in \bigl(f(\alpha_t(H))\cap L\bigl)\bigl(\Spec\bigl(\FF\drpt\bigr)\bigr).
\]
 
Thus over~$\FF\drpt$ we have
\begin{align*}
\frac{d^{n-1}e}{\deg(f|_{\alpha_t(H)})}=\deg(f(\alpha_t(H)))\ge \abs{f(\alpha_t(H))\cap L} \ge \ddd+1 > \ddd .
\end{align*}
Since~$\ddd=d^{n-1}e/r$, this gives a strict inequality
\begin{align*}
\deg(f|_{\alpha_t(H)}) < r,
\end{align*}
which completes the proof of Lemma~$\ref{lemma:rto1tosto1}$ over the
algebraically closed characteristic~$0$ field~$\FF$.
\end{proof}

\begin{lemma}  
\label{lem:taillengthtwo}
Let~$n\ge3$ and~$d\ge3$ and $\ell=2$. Then there exists an~$f_0\in\End_d^n$ that
satisfies Conditions \textup{(1)--(4)} of
Proposition~$\ref{prop:taillengthell}$.  
\end{lemma}
\begin{proof}
By Proposition~\ref{proposition:hyperplane} and
Theorem~\ref{theorem:critgentype}, there exists $f\in\End_d^n$ such
that
\begin{parts}
  \Part{\textbullet}
  $\Crit_f$ is irreducible and of general type.
  \Part{\textbullet}
  $f$ is not PCF with tail length $1$, i.e., 
  \[
f(\Branch_f)=f^2(\Crit_f)\not\subset f(\Crit_f)=\Branch_f.
\]
 \Part{\textbullet}
  $f:\Crit_f\to \Branch_f$ is generically~$1$-to-$1$.
\Part{\textbullet}
  $f$ has multiplicity $2$ along $\Crit_f$
\end{parts}
Thus $f$ satisfies conditions (1), (2) and (4) of the hypotheses of
Proposition~\ref{prop:taillengthell}. If~$f|_{\Branch_f}$ is
generically~$1$-to-$1$, then $f$ also satisfies condition (3) so we
are done. If not, we use Lemma~\ref{lemma:rto1tosto1} to find
an~$\alpha\in\PGL_{n+1}$ such that~$f|_{\alpha(\Branch_f)}$ is
generically~$1$-to-$1$. Set $f_0=\alpha\circ f$. Then
\[
\Crit_{f_0}=\Crit_f,\quad
\Branch_{f_0}=\alpha(\Branch_f),
\quad\text{and}\quad
(f_0)|_{\Branch_{f_0}}=(f_0)|_{\alpha(\Branch_f)}=(\alpha\circ f)|_{\alpha(\Branch_f)}.
\]
This last map~$(f_0)|_{\Branch_{f_0}}$ is generically~$1$-to-$1$
because~$f|_{\alpha(\Branch_f)}$ is generically~$1$-to-$1$
and~$\alpha$ is everywhere~$1$-to-$1$. Finally the multiplicity
of~$f_0$ equals the multiplicity of~$f$ along~$\Crit_{f_0}=\Crit_f$,
thus is~$2$. Thus~$f_0$ satisfies the hypotheses of
Proposition~\ref{prop:taillengthell} for~$\ell=2$.
\end{proof}

We now have the tools to prove the main result of this section, which
is that PCF maps with tail length at most~$2$ are sparse.

\begin{theorem} 
\label{thm:Preperiodtwo}
Let~$n\ge3$ and~$d\ge3$. Then
\[
\{f\in \End_d^n :\text{$f^{k}(\Crit_f)\subseteq f^{2}(\Crit_f)$ for some $k\ge2$} \}
\]
is contained in a proper closed subvariety of~$\End_d^n$.
\end{theorem}
\begin{proof}
By Lemma~\ref{lem:taillengthtwo}, there exists a map $f_0\in\End_d^n$
satisfying the hypotheses of Proposition~\ref{prop:taillengthell} for
$\ell=2$. Thus we can use Proposition~\ref{prop:taillengthell} to
conclude that
\[
  \{f\in \End_d^n :\text{$f^{k}(\Crit_f)\subseteq
  f^{2}(\Crit_f)$ for some $k\ge2$} \}
\]
is contained in a proper closed subvariety of~$\End_d^n$.
\end{proof}


\begin{acknowledgement}
We thank MathOverflow users for pointing us towards the
references~\cite{MR2775814,MR1817250} used in
Theorem~\ref{theorem:A2case}
(\url{mathoverflow.net/q/163086/11926}). We thank Jason Starr for
noting that for~$n\ge3$, the critical locus~$\Crit_f$ should be
generically singular, since the determinant locus is itself singular
in codimension~3, and for showing us the proof of
Theorem~\ref{theorem:critgentype} which says that despite the
singularities, the critical locus~$\Crit_f$ is generically a variety
of general type.  \par Ingram and Silverman thank AIM for hosting a
workshop on Postcritically Finite Maps in 2014 during which this
research was initiated. All three authors thank ICERM for hosting them
at a week-long collaboration in~2019, which enabled them to make
significant further progress on the results contained in this paper.
\end{acknowledgement}



\begin{thebibliography}{10}

\bibitem{MR2581826}
J.~Belk and S.~Koch.
\newblock Iterated monodromy for a two-dimensional map.
\newblock In {\em In the tradition of {A}hlfors-{B}ers. {V}}, volume 510 of
  {\em Contemp. Math.}, pages 1--11. Amer. Math. Soc., Providence, RI, 2010.

\bibitem{MR1940161}
A.~M. Bonifant and M.~Dabija.
\newblock Self-maps of {$\mathbb{P}^2$} with invariant elliptic curves.
\newblock In {\em Complex manifolds and hyperbolic geometry ({G}uanajuato,
  2001)}, volume 311 of {\em Contemp. Math.}, pages 1--25. Amer. Math. Soc.,
  Providence, RI, 2002.

\bibitem{MR3762687}
A.~Bridy, P.~Ingram, R.~Jones, J.~Juul, A.~Levy, M.~Manes,
  S.~Rubinstein-Salzedo, and J.~H. Silverman.
\newblock Finite ramification for preimage fields of post-critically finite
  morphisms.
\newblock {\em Math. Res. Lett.}, 24(6):1633--1647, 2017.

\bibitem{MR2775814}
C.~Ciliberto and F.~Flamini.
\newblock On the branch curve of a general projection of a surface to a plane.
\newblock {\em Trans. Amer. Math. Soc.}, 363(7):3457--3471, 2011.

\bibitem{MR2603592}
M.~Dabija and M.~Jonsson.
\newblock Algebraic webs invariant under endomorphisms.
\newblock {\em Publ. Mat.}, 54(1):137--148, 2010.

\bibitem{MR3831028}
L.~De~Marco.
\newblock Dynamical moduli spaces and elliptic curves.
\newblock {\em Ann. Fac. Sci. Toulouse Math. (6)}, 27(2):389--420, 2018.
\newblock {KAWA} 2015 Pisa.

\bibitem{MR1251582}
A.~Douady and J.~H. Hubbard.
\newblock A proof of {T}hurston's topological characterization of rational
  functions.
\newblock {\em Acta Math.}, 171(2):263--297, 1993.

\bibitem{MR1190986}
J.~E. Forn\ae~ss and N.~Sibony.
\newblock Critically finite rational maps on {${\textbf{P}}^2$}.
\newblock In {\em The {M}adison {S}ymposium on {C}omplex {A}nalysis ({M}adison,
  {WI}, 1991)}, volume 137 of {\em Contemp. Math.}, pages 245--260. Amer. Math.
  Soc., Providence, RI, 1992.

\bibitem{MR2509692}
Y.~Fujimoto and N.~Nakayama.
\newblock Complex projective manifolds which admit non-isomorphic surjective
  endomorphisms.
\newblock In {\em Higher dimensional algebraic varieties and vector bundles},
  RIMS K\^oky\^uroku Bessatsu, B9, pages 51--79. Res. Inst. Math. Sci. (RIMS),
  Kyoto, 2008.

\bibitem{MR3034294}
C.~D. Hacon, J.~McKernan, and C.~Xu.
\newblock On the birational automorphisms of varieties of general type.
\newblock {\em Ann. of Math. (2)}, 177(3):1077--1111, 2013.

\bibitem{hartshorne}
R.~Hartshorne.
\newblock {\em Algebraic {G}eometry}, volume~52 of {\em Graduate Texts in
  Mathematics}.
\newblock Springer-Verlag, New York, 1977.

\bibitem{MR637060}
S.~Iitaka.
\newblock {\em Algebraic geometry}, volume~76 of {\em Graduate Texts in
  Mathematics}.
\newblock Springer-Verlag, New York-Berlin, 1982.
\newblock An introduction to birational geometry of algebraic varieties,
  North-Holland Mathematical Library, 24.

\bibitem{MR3492630}
P.~Ingram.
\newblock Rigidity and height bounds for certain post-critically finite
  endomorphisms of {$\mathbb{P}^N$}.
\newblock {\em Canad. J. Math.}, 68(3):625--654, 2016.

\bibitem{MR1609475}
M.~Jonsson.
\newblock Some properties of {$2$}-critically finite holomorphic maps of
  {$\textbf{P}^2$}.
\newblock {\em Ergodic Theory Dynam. Systems}, 18(1):171--187, 1998.

\bibitem{MR3107522}
S.~Koch.
\newblock Teichm\"uller theory and critically finite endomorphisms.
\newblock {\em Adv. Math.}, 248:573--617, 2013.

\bibitem{MR1817250}
V.~S. Kulikov and V.~S. Kulikov.
\newblock Generic coverings of the plane with {$A$}-{$D$}-{$E$}-singularities.
\newblock {\em Izv. Ross. Akad. Nauk Ser. Mat.}, 64(6):65--106, 2000.

\bibitem{MR2741188}
A.~Levy.
\newblock The space of morphisms on projective space.
\newblock {\em Acta Arith.}, 146(1):13--31, 2011.

\bibitem{MR3199801}
V.~Nekrashevych.
\newblock Combinatorial models of expanding dynamical systems.
\newblock {\em Ergodic Theory Dynam. Systems}, 34(3):938--985, 2014.

\bibitem{northcott:periodicpoints}
D.~G. Northcott.
\newblock Periodic points on an algebraic variety.
\newblock {\em Ann. of Math. (2)}, 51:167--177, 1950.

\bibitem{MR2567424}
C.~Petsche, L.~Szpiro, and M.~Tepper.
\newblock Isotriviality is equivalent to potential good reduction for
  endomorphisms of {$\mathbb{P}^N$} over function fields.
\newblock {\em J. Algebra}, 322(9):3345--3365, 2009.

\bibitem{MR2415046}
F.~Rong.
\newblock The {F}atou set for critically finite maps.
\newblock {\em Proc. Amer. Math. Soc.}, 136(10):3621--3625, 2008.

\bibitem{MR2316407}
J.~H. Silverman.
\newblock {\em The {A}rithmetic of {D}ynamical {S}ystems}, volume 241 of {\em
  Graduate Texts in Mathematics}.
\newblock Springer, New York, 2007.

\bibitem{MR2884382}
J.~H. Silverman.
\newblock {\em Moduli {S}paces and {A}rithmetic {D}ynamics}, volume~30 of {\em
  CRM Monograph Series}.
\newblock American Mathematical Society, Providence, RI, 2012.

\bibitem{arxiv0305432}
J.~Starr.
\newblock The {K}odaira dimension of spaces of rational curves on low degree
  hypersurfaces, 2003.
\newblock \url{arxiv.org/abs/math/0305432}.

\bibitem{MR1747345}
T.~Ueda.
\newblock Critically finite maps on projective spaces.
\newblock Number 1087, pages 132--138. 1999.
\newblock Research on complex dynamical systems: current state and prospects
  (Japanese) (Kyoto, 1998).

\bibitem{MR738261}
I.~Vainsencher.
\newblock Complete collineations and blowing up determinantal ideals.
\newblock {\em Math. Ann.}, 267(3):417--432, 1984.

\end{thebibliography}

\def\cprime{$'$}

\end{document}